\documentclass[10pt, leqno, a4paper]{amsart}

\usepackage[showonlyrefs]{mathtools}
\mathtoolsset{showonlyrefs=true} 

\usepackage{amsmath, amsthm}
\usepackage{mathrsfs}    
\usepackage{amssymb}

\usepackage{graphicx,color}

\usepackage{hyperref}
\usepackage{verbatim}

\newtheorem{theorem}{Theorem}[section]
\newtheorem{proposition}[theorem]{Proposition}
\newtheorem{corollary}[theorem]{Corollary}
\newtheorem{lemma}[theorem]{Lemma}
\newtheorem*{theorem_a}{Theorem A}
\newtheorem*{theorem_b}{Theorem B}
\newtheorem*{theorem_c}{Theorem C}
\newtheorem*{theorem_d}{Theorem D}

\theoremstyle{definition}
\newtheorem{definition}[theorem]{Definition}
\newtheorem{remark}[theorem]{Remark}

\usepackage[
  hmarginratio={1:1},     
  vmarginratio={1:1},     
  textwidth=15cm,        
  textheight=20cm,  
  heightrounded,          
]{geometry}

\numberwithin{equation}{section}

\def\a{\alpha}
\def\b{\beta}

\def\eps{\varepsilon}
\def\ep{\varepsilon}

\def\div{\textrm{div}}

\def\O{\Omega}

\def\Bcal{\mathcal{B}}
\def\Ccal{\mathcal{C}}

\def\Gcal{\mathcal{G}}
\def\Kcal{\mathcal{K}}

\def\Mcal{\mathcal{M}}

\def\Rcal{\mathcal{R}}
\def\Ucal{\mathcal{U}}

\def\Hh{\mathscr{H}}

\def\N{\mathbb{N}}

\def\R{\mathbb{R}}

\def\dist{\mathrm{dist}}

\def\hbar{\bar{h}}

\def\le{\leqslant}
\def\leq{\leqslant}
\def\geq{\geqslant}

\newcommand{\de}[1] {\mathrm{d} #1}

\def\pa{\partial}

\newcommand{\rst}[1]{\ensuremath{{\mathbin |}%
\raise-.5ex\hbox{$#1$}}}
\newcommand{\mf}[1]{\mathbf{#1}}

\newcommand{\loc}{\mathrm{loc}}

\definecolor{verde}{rgb}{0,0.35,0.1} 
\definecolor{rosso}{rgb}{0.7,0,0}
\definecolor{blue}{rgb}{0,0,1}
\definecolor{viola}{rgb}{0.6,0,0.4}
\definecolor{grigio}{rgb}{0.5,0.5,0.5}

\author[N. Soave]{Nicola Soave}\thanks{}
\address{Nicola Soave \newline \indent
Mathematisches Institut, \newline \indent Justus-Liebig-Universit\"at Giessen, \newline \indent
Arndtstrasse 2, 35392 Giessen, Germany}
\email{nicola.soave@gmail.com; nicola.soave@math.uni-giessen.de}

\author[H. Tavares]{Hugo Tavares}\thanks{}
\address{Hugo Tavares \newline \indent  Center for Mathematical Analysis, Geometry and Dynamical Systems \newline \indent Mathematics Department \newline \indent Instituto Superior T\'ecnico, Universidade de Lisboa \newline \indent Av. Rovisco Pais, 1049-001 Lisboa, Portugal}
\email{htavares@math.ist.utl.pt}

\author[S. Terracini]{Susanna Terracini}\thanks{}
\address{Susanna Terracini \newline \indent
 Dipartimento di Matematica ``Giuseppe Peano'', \newline \indent
Universit\`a di Torino, \newline \indent
Via Carlo Alberto, 10,
10123 Torino, Italy}
\email{susanna.terracini@unito.it}

\author[A. Zilio]{Alessandro Zilio}\thanks{}
\address{Alessandro Zilio \newline \indent Centre d'analyse et de math\'{e}matique sociales\newline 	\indent
\'{E}cole des Hautes \'{E}tudes en Sciences Sociales \newline 	\indent
190-198 Avenue de France, 75244, Paris CEDEX 13, France}
\email{azilio@ehess.fr, alessandro.zilio@polimi.it}

\title[Uniform H\"older bounds and regularity of the emerging free boundaries]{H\"older bounds and regularity of emerging free boundaries for strongly competing
Schr\"odinger equations with nontrivial grouping}

\subjclass[2010]{35B45, 35B65, 35R35 (Primary), 35B53, 35B25, 35J47, 35R06}

\keywords{Nonlinear Schr\"odinger systems, Bounds in H\"older Norm, Segregation between groups, Regularity of Nodal Sets, Weak reflection law, Liouville type theorems}

\thanks{Hugo Tavares is partially supported by Funda\c c\~ao para a Ci\^encia e Tecnologia through the program ``Investigador FCT'' and the project UID/MAT/04459/2013. Alessandro Zilio is partially supported by the ERC Advanced Grant 2013 n. 321186 ``ReaDi -- Reaction-Diffusion Equations, Propagation and Modelling''. The authors are partially supported through the project ERC Advanced Grant 2013 n. 339958 ``Complex Patterns for Strongly Interacting Dynamical Systems - COMPAT". }

\begin{document}

\begin{abstract}
We study interior regularity issues for systems of elliptic equations of the type
\[
	-\Delta u_i=f_{i,\beta}(x)-\beta \sum_{j\neq i} a_{ij} u_i |u_i|^{p-1}|u_j|^{p+1}
\]
set in domains $\Omega \subset \R^N$, for $N \geq 1$. The paper is devoted to the derivation of $\Ccal^{0,\alpha}$ estimates that are uniform in the competition parameter $\beta > 0$, as well as to the regularity of the limiting free-boundary problem obtained for $\beta \to + \infty$.

The main novelty of the problem under consideration resides in the non-trivial grouping of the densities: in particular, we assume that the interaction parameters $a_{ij}$ are only non-negative, and thus may vanish for specific couples $(i,j)$. As a main consequence, in the limit $\beta \to +\infty$, densities do not segregate pairwise in general, but are grouped in classes which, in turn, form a mutually disjoint partition. Moreover, with respect to the literature, we consider more general forcing terms, sign-changing solutions, and an arbitrary $p>0$. In addition, we present a regularity theory of the emerging free-boundary, defined by the interface among different segregated groups.

These equations are very common in the study of Bose-Einstein condensates and are of key importance for the analysis of optimal partition problems related to high order eigenvalues.

\end{abstract}

\maketitle

\begin{center}
\emph{Dedicated to Prof. Juan Luis V\'azquez with deep admiration and respect}
\end{center}

\tableofcontents

\section{Introduction}

The asymptotic behaviour of solutions of competing systems in the limit of strong competition has been object of an intense research in the last decades. A well known example is represented by 
\begin{equation}\label{classical elliptic GP}
\begin{cases}
-\Delta u_i + \lambda_i u_i= \mu_i u_i^3 - \beta u_i \sum_{j \neq i} a_{ij} u_j^2 & \text{in $\Omega$} \\
u_i =0 & \text{on $\pa \Omega$}
\end{cases} \qquad i=1,\dots,d,
\end{equation}
where $\Omega$ is a smooth domain of $\R^N$, $\lambda_i,\mu_i \in \R$, and $a_{ij}=a_{ji} >0$. 

System \eqref{classical elliptic GP} naturally arises in several contexts: from physical applications, it is obtained in the search of solitary waves for the corresponding system of Schr\"odinger equations, which is of interest in nonlinear optics and in the Hartree-Fock approximation for Bose-Einstein condensates with multiple hyperfine states, see e.g. \cite{AkAn, Timm}. From a purely mathematical point of view, \eqref{classical elliptic GP} is useful in the approximation of optimal partition problems for Laplacian eigenvalues, as well as in the theory of harmonic maps into singular manifolds, see \cite{CaffLin, CoTeVe2002,CoTeVe2003, rtt,TaTePoin}. A common feature in the previous situations resides in the fact that one has to deal with different densities $u_i$ living in a domain $\Omega$ and subject to diffusion ($-\Delta u_i$), reaction ($\mu_i u_i^3- \lambda_i u_i$), and mutual interaction ($\beta u_i \sum_{j \neq i}  a_{ij} u_j^2$). As we shall see, in  addition to the different values of $\lambda_i$ and $\mu_i$, a crucial role is played by the coupling parameters $\beta \cdot a_{ij}$, which describe the interaction between the densities $u_i$ and $u_j$: with the previous sign convention, if $\beta<0$, then $u_i$ cooperates with $u_j$, while if $\beta>0$, then $u_i$ competes with $u_j$; moreover, the larger is $|\beta|$, the stronger is the strength of the interaction. Notice that the condition $a_{ij}=a_{ji}$ reflects the symmetry of the inter-species relations and, throughout this paper, constitutes a crucial assumption. 

It is quite easy to understand why $a_{ij}=a_{ji}$ is crucial from the point of view of the existence of solutions. Indeed, if it is fulfilled, solutions of \eqref{classical elliptic GP} are critical points of the functional $J:H_0^1(\Omega,\R^d) \to \R$, defined by  
\[
J(\mf{u}) := \int_{\Omega} \left[ \frac{1}{2} \sum_{i=1}^d  \left( |\nabla u_i|^2 + \lambda_i u_i^2 - \frac{1}{2} \mu_i u_i^4 \right) + \frac{\beta}{4} \sum_{i \neq j} a_{ij} u_i^2 u_j^2 \right],
\]
where we used the vector notation $\mf{u}:=(u_1,\dots,u_d)$. This variational structure in dimension $N\leq 3$ or $N=4$ has been exploited in order to obtain several existence and multiplicity results. A complete review of these is out of the aims of the present work; we refer for instance to the introduction of \cite{So} (see also the references therein), and we only restrict ourselves to recall that under the assumption $\beta \ge 0$ system \eqref{classical elliptic GP} has infinitely many solutions, obtained by minimax argument. 
The variational characterization of these solutions implies energy bounds independent of $\beta$, which in turn give uniform bounds in the $H^1$ norm. In turn, recalling the definition of $J$, we obtain uniform bounds for the interaction terms
\[
\beta \int_{\Omega} u_{i,\beta}^2 u_{j,\beta}^2 \le C \qquad  \forall \beta, \ \forall i \neq j,
\]
and, taking the limit as $\beta \to +\infty$, we infer that, for the considered family of solutions, it results
\begin{equation}\label{point-wise separation}
 u_{i,\beta} u_{j,\beta} \to 0 \qquad \text{a.e. in $\Omega$},
\end{equation}
that is, in the limit of strong competition, different densities tend to assume disjoint supports. This phenomenon is called \emph{phase-separation}.

At this point a number of natural questions arise, such as:
\begin{itemize}
\item[($i$)] is it possible to develop a common regularity theory for the families of solutions of \eqref{classical elliptic GP} as $\beta \to +\infty$?
\item[($ii$)] In addition to \eqref{point-wise separation}, can we say that the sequence $\{(u_{1,\beta}, \dots, u_{k,\beta})\}$ converges to a limiting profile in some topology? 
\item[($iii$)] If the answer to ($ii$) is affirmative, what are the properties of the limiting profile? 
\end{itemize}
As we shall see, for positive solutions of system \eqref{classical elliptic GP} the picture is now well understood, and optimal results are available. The purpose of this manuscript, which can be considered as an intermediate step between an original research paper and a survey, is the generalization of these results in several different directions. 

\subsection{Review of known results}

Let us now review the results which are already available for problem \eqref{classical elliptic GP}; all of them concern \emph{positive solutions}. The first contributions can be ascribed to Conti et al. \cite{CoTeVe2002,CoTeVe2003}, where the authors proved that sequences of constrained minimizers associated to variational problem of type \eqref{classical elliptic GP} with $\mu_i>0$ converge in $H^1(\Omega)$, as $\beta \to +\infty$, to a segregated configuration (actually they considered a slightly different problem, but once the existence of solutions is settled, their asymptotic analysis works perfectly for \eqref{classical elliptic GP}). The case $\mu_i<0$ has been first studied by Chang et. al. in \cite{clll}, where point-wise phase-separation is proved. 

A new approach, based on the use of some Almgren-type monotonicity formulae for elliptic systems, has been later introduced in \cite{CaffLin}, where Caffarelli and Lin have shown the $\mathcal{C}^{0,\alpha}$-convergence of families of minimizers associated to \eqref{classical elliptic GP} with $\lambda_i=\omega_i=0$, and with non-homogeneous boundary conditions. This fundamental result, which rests in an essential way on the minimality of the solutions, has been generalized to excited states of \eqref{classical elliptic GP} with any $\lambda_i \in \R$ and $\omega_i \in \R$ by Noris et al. in \cite{nttv}. To be precise, the authors proved the following:

\begin{theorem_a}\label{theorem A}
Let $\Omega$ be a bounded smooth domain of $\R^N$ with $N \le 3$, let us assume that $a_{ij}=a_{ji}$, $\mu_i \in \R$ and that $\{\lambda_i= \lambda_{i,\beta}\}$ is a bounded sequence. Let $\{\mf{u}_{\beta}\} \subset H_0^1(\Omega)$ be a family of positive solutions of \eqref{classical elliptic GP}, uniformly bounded in $L^\infty(\Omega)$. Then for every $0<\alpha<1$ there exists $M>0$, independent of $\beta$ such that
\[
\|\mf{u}_{\beta}\|_{\mathcal{C}^{0,\alpha}(\overline{\Omega})} \le M.
\]
\end{theorem_a}
Previously, under the same assumptions Wei and Weth \cite{WeWeth} proved the equi-continuity of $\{\mf{u}_\beta\}$ in dimension $N=2$. We recall that in \cite{WeWeth} a very general class of systems is considered. In particular, to our knowledge, this is the only available research paper which treats the case $a_{ij} \neq a_{ji}$.

It is worth to mention that the assumption ``$\{\mf{u}_\beta\}$ is uniformly bounded in $L^\infty(\Omega)$" is very weak. Indeed, by elliptic regularity, it turns out that if we have a common energy bound of type $J(\mf{u}_\beta) \le C$ and $\{\lambda_{i,\beta}\}$ is bounded, then the assumption is satisfied. Therefore, for instance in Theorem A one can consider families of possibly excited states sharing a common energy bound. 

It is also important to observe that a deep analysis of the proof of Theorem A reveals that it is valid as it is stated also in dimension $N=4$. This has been used for instance in the paper by Chen and Zou \cite{ChenZou}, where a Brezis--Nirenberg type problem is tackled. Under the additional assumption $\lambda_{i,\beta} \ge0$, $\omega_i \le 0$, Theorem A works in any dimension $N \ge 1$ (we refer to Remark 3.4 in \cite{SoZi}).

\medskip

Regarding the consequences of the uniform $\mathcal{C}^{0,\alpha}$-boundedness, we observe that this implies, up to a subsequence, convergence to a nonnegative limit $\mf{u}$ in $\mathcal{C}^{0,\alpha}(\Omega)$, for every $0<\alpha<1$. Moreover, since $\lambda_{i,\beta}$ is bounded, we can suppose that along such sequence $\lambda_{i,\beta} \to \lambda_{i,\infty}$. In \cite{nttv}, the authors proved the basic properties of $\mf{u}$.

\begin{theorem_b}\label{thm: consequences nttv}
In the previous setting, we have:
\begin{enumerate}
\item $\mf{u}_\beta \to \mf{u}$ strongly in $H^1(\Omega)$, and
\[
\int_{\Omega} \beta u_i^2 u_j^2 \to 0
\]
as $\beta \to +\infty$, for every $i \neq j$;
\item $u_i$ is Lipschitz continuous in $\Omega$;
\item $u_i u_j\equiv 0$ whenever $i \neq j$ (segregation between components);
\item for each $i=1,\dots,d$ it results that
\[
-\Delta u_i=\mu_i u_i^3-\lambda_{i,\infty} u_i \qquad \text{ in the open set} \left\{u_i>0\right\}.
\]
\end{enumerate}
\end{theorem_b}

Theorems A and B have been extended to a local formulation in \cite[Theorem 2.6]{Wa}: to be precise, it is proved that if the assumption of Theorem A is satisfied in a domain $\Omega$ (neither necessarily bounded, nor smooth), then for any compact set $K \Subset \Omega$ the family $\{\mf{u}_\beta\}$ is uniformly bounded in $\mathcal{C}^{0,\alpha}(K)$, for every $0<\alpha<1$. This result turns out to be extremely useful in blow-up analysis or similar contexts, when one has to deal with sequences of functions defined on varying domains, and hence the global estimate of Theorem A would not be applicable. Moreover, one can also prove local estimates up to the boundary, under some regularity assumption on the domain $\Omega$ (thus recovering global results for $\Omega$ bounded and smooth). 

\medskip

Since each $u_i$ solves an elliptic equation in its positivity domain, by Hopf lemma the Lipschitz continuity of $u_i$ is optimal. One could then wonder if it is possible to improve the result in \cite{nttv}, establishing uniform boundedness of $\{\mf{u}_\beta\}$ in Lipschitz norm, which would be optimal. This result has been proved recently in local form in \cite{SoZi}. We refer also to \cite[Lemma 2.4]{BeLiWeZh}, where the $1$-dimensional case in the interval $[0,1]$ is considered, and fine properties of the phase separation are derived using the Lipschitz boundedness (H\"older bounds would not be sufficient for this purpose). We refer to \cite{SoZi2} for the corresponding analysis in higher dimension.

\medskip

We have seen that limit profiles of solutions to \eqref{classical elliptic GP} are segregated configurations. It is then natural to define the \emph{free-boundary} as the nodal set $\Gamma_{\mf{u}}:= \{\mf{u}=\mf{0}\}$. The regularity of the free-boundary has been studied in \cite{CaffLin} under the assumptions that $\{\mf{u}_\beta\}$ is a family of minimizers for $J$ with $\mu_i=\lambda_i=0$; the results in \cite{CaffLin} have been applied by the authors to the study of an optimal partition problem involving sums of first Dirichlet eigenvalues \cite{CaffLin2007}. Further informations about the structure of the singular set has been provided in \cite{CaffLin2010}. Concerning non-minimal solutions, we refer to \cite{tt}, where a very general class of functions, including all the limits coming from Theorems A and B, is treated, and to \cite{ZhangLiu}, which extends the results in \cite{CaffLin2010} to the setting considered in \cite{tt}. Let us review in detail the results in \cite{tt}.

\begin{definition}[Definition 1.2 in \cite{tt}]\label{def:old_class_G}
We define $\mathcal{G}(\Omega)$ as the set of functions $\mf{u}=(u_1,\dots,u_d) \in H^1(\Omega, \R^d) \setminus \{\mf{0}\}$ such that:
\begin{itemize}
\item[(G1)] $u_i$ are nonnegative, Lipschitz continuous on $\Omega$, and such that $u_i u_j \equiv 0$ in $\Omega$ for every $i \neq j$;
\item[(G2)] each component $u_i$ satisfies
\[
-\Delta u_i = f_i(x,u_i) - \mathcal{M}_i \quad \text{in $\mathcal{D}'(\Omega)= (\mathcal{C}^\infty_c(\Omega))'$},
\]
where we suppose that there exists $C>0$ such that
\[
\sup_{s \in [0,1]} \sup_x  \left|\frac{f_i(x,s)}{ |s|} \right|  \le C
\]
for every $i=1,\dots,k$, and $\mathcal{M}_i$ are nonnegative Radon measures supported on $\Gamma_{\mf{u}}$.
\item[(G3)] for every $x_0 \in \Omega$ and $0<r<\dist(x_0,\partial \Omega)$ it holds
\begin{align*}
(2-N) \sum_{i=1}^d \int_{B_r(x_0)} |\nabla u_i|^2 &= r \sum_{i=1}^d \int_{\pa B_r(x_0)} \left( 2 (\pa_{\nu} u_i)^2 -|\nabla u_i|^2 \right) \\
& \qquad + 2 \sum_{i=1}^d \int_{B_r(x_0)} f_{i}(x,u_i) \nabla u_i \cdot (x-x_0).
\end{align*}
\end{itemize}
We write that $\mf{u} \in \mathcal{G}_{\loc}(\R^N)$ if $\mf{u} \in \mathcal{G}(B_R(0))$ for every $R>0$.
\end{definition}

Notice that (G3) is not stated as in \cite{tt}, but it is not difficult to check that the two formulations are equivalent. In the following regularity result, which corresponds to Theorem 1.1 in \cite{tt}, $\Hh_\text{dim}(A)$ denotes the Hausdorff dimension of $A$.

\begin{theorem_c}
Let $\mf{u}\in \mathcal{G}(\Omega)$. Then
\begin{itemize}
\item[1.] $\Hh_\text{dim}(\Gamma_\mf{u})\leq N-1$; 
\item[2.] there exists a set $\Rcal_\mf{u}\subseteq \Gamma_\mf{u}$, relatively open in $\Gamma_\mf{u}$, such that
\begin{itemize}
\item[-] $\Hh_\text{dim} (\Gamma_\mf{u}\setminus \Rcal_\mf{u})\leq N-2$;
\item[-] $\Rcal_\mf{u}$ is a collection of  hypersurfaces of class $\Ccal^{1,\alpha}$ (for some $0<\alpha<1$), each one locally separating two connected components of $\Omega\setminus \Gamma_{\mf{u}}$.
\item[-] given $x_0\in \Rcal_\mf{u}$, there exist $i,j \in \{1,\dots,k\}$ such that
\[
\lim_{x\to x_0^+} |\nabla u_i|^2=\lim_{x\to x_0^-}  |\nabla u_j|^2\neq 0,
\]
where $x\to x_0^\pm$ are limits taken from opposite sides of the hypersurface.
\item[-] whenever $x_0\in \Gamma_\mf{u}\setminus \Rcal_\mf{u}$, we have
\[
\sum_{i=1}^d |\nabla u_i(x)|^2\to 0 \qquad \text{ as } x\to x_0.
\]
\end{itemize}
\item[3.] Furthermore, if $N=2$, then $\Rcal_{\mf{u}}$ consists in a locally finite collection of curves meeting with equal angles at singular points.
\end{itemize}
\end{theorem_c}

In the context of phase-separation for strongly competing systems, the previous result allows to describe the regularity properties of any limit profile, as established by Theorem 8.1 in \cite{tt}.

\begin{theorem_d}
Under the assumptions of Theorem A, let $\mf{u}$ be a limit of $\{\mf{u}_\beta\}$ as $\beta \to +\infty$, and suppose that $u_i \not \equiv 0$ in $\Omega$ for some $i$. Then $\mf{u}\in\Gcal(\Omega)$. In particular, the nodal set of the limit profile satisfies all the conclusions of Theorem C.
\end{theorem_d}

\subsection{The problem under investigation}

In this paper we aim at generalizing Theorems A, B, C and D in a very general setting. To be precise, we have in mind to approach the following issues:
\begin{itemize}
\item[($i$)] all the previous results concern positive solutions but, expecially when dealing with excited states, one would like to treat sign-changing solutions as well;
\item[($ii$)] we think that it can be interesting, for modelling and theoretical reasons, to replace the nonlinear term $\mu_i u_i^3-\lambda_i u_i$ with a general term of type $f_i(x,u_i)$, possibly depending on $\beta$;
\item[($iii$)] it is natural, in general, to replace the interaction terms $u_i u_j^2$ in \eqref{classical elliptic GP} with a more general power law of type $u_i|u_i|^{p-1}  |u_j|^{p+1}$, with $p>0$ (which might be sublinear in $u_i$);
\item[($iv$)] assuming $a_{ij}=a_{ji}>0$ and $\beta>0$, we restrict ourselves to a purely competitive setting. What happens if we allow some $a_{ij}$ to be zero, inducing segregation between groups of components, and if we have mixed cooperation and competition?
\end{itemize}

We mention that phase-separation in systems with non-trivial grouping has been already studied in particular cases in \cite{CaffLin,rtt,So}. In \cite{CaffLin,So} minimal solutions are considered, while in \cite{rtt} systems corresponding to singular perturbations of eigenvalue problems are studied.

\medskip

To state our results in full generality, we introduce some notation. For an arbitrary $m \leq d$, we say that a vector $\mf{a}=(a_0,\dots,a_m) \in \N^{m+1}$ is an \emph{$m$-decomposition of $d$} if
\[
0=a_0<a_1<\dots<a_{m-1}<a_m=d;
\]
given a $m$-decomposition $\mf{a}$ of $d$, we set, for $h=1,\dots,m$,
\begin{equation}\label{def indexes}
\begin{split}
& I_h:= \{i \in  \{1,\dots,d\}:  a_{h-1} < i \le a_h \}, \\
&\Kcal_1:= \left\{(i,j) \in I_h^2 \text{ for some $h=1,\dots,m$, with $i \neq j$}\right\}, \\ &\mathcal{K}_2 := \left\{(i,j) \in I_h \times I_k \text{ with $h \neq k$} \right\}. 
\end{split}
\end{equation}
This way, we have partitioned the set $\{1,\ldots, d\}$ into $m$ groups $I_1,\ldots, I_m$.
We will consider the system for $\mf{u}=(u_1,\ldots, u_d)$
\begin{equation}\label{eq:main_system}
-\Delta u_i=f_{i,\b}-\beta \mathop{\sum_{j=1}^d}_{j\neq i} a_{ij} u_i |u_i|^{p-1}|u_j|^{p+1} \quad \text{ in } \Omega,\qquad i=1,\ldots, d.\\
\end{equation}
with $\beta>0$, $p > 0$, $a_{ij}=a_{ji}$, being $a_{ij}=0$ for $(i,j)\in \Kcal_1$, $a_{ij}>0$ whenever $(i,j)\in \Kcal_2$. This basically means that the term
\[
\beta \mathop{\sum_{j=1}^d}_{j\neq i} a_{ij} u_i |u_i|^{p-1}|u_j|^{p+1}
\]
represents a competing term between groups of components: heuristically speaking, $u_i$ and $u_j$ compete if $i\in I_{h}$ and $j\in I_{k}$ for $h\neq k$. The assumption on the nonlinear terms $f_{i,\beta}$ depends on the value of $p$.
\begin{itemize}
\item[(H)] If $p \ge 1$, then $f_{i,\beta}:\Omega \times \R^d \to \R$, and given $K\Subset \Omega \times \R^d$ there exists $C=C(K)$ such that
\[
|f_{i,\beta}(x,\mf{s})|\leq C \qquad \forall i=1,\ldots, d,\ (x,\mf{s})\in K.
\]
If $0<p<1$, then $f_{i,\beta}:\Omega\times \R^d \to \R$, and we suppose that given $K\Subset \Omega$ there exists $C=C(K)$ such that
\[
|f_{i,\beta}(x,\mf{s})|\leq C \sum_{j\in I_h}  |s_j|^p\qquad \forall i\in I_h, (x,\mf{s})\in K\times \R^d.
\]
\end{itemize}

We are interested in the asymptotic behaviour, as $\beta\to +\infty$, of families of possibly sign-changing solutions $\{\mf{u}_{\beta}\}$. More precisely, the following theorem states that, locally, uniform $L^\infty$ bounds imply uniform $\Ccal^{0,\alpha}$ bounds, for every $0<\a<1$.

\begin{theorem}\label{thm: holder bounds}
Let $N \ge 1$, $p >0$, $\mf{a}$ be a $m$-decomposition, and assume that $\mf{f}_{\beta}$ satisfies (H). Let $\{\mf{u}_\b\}_\b$ be a family of solutions of \eqref{eq:main_system}, uniformly bounded in $L^\infty(\Omega)$. Then for every $\Omega' \Subset \Omega$ and $\alpha\in (0,1)$ there exists $C=C(\Omega',\a)>0$ such that
\[
\| \mf{u}_\b\|_{\Ccal^{0,\alpha}(\Omega')}\leq C.
\]
\end{theorem}

Notice that, due to the local nature of the result, we require neither the boundedness, nor the regularity of $\Omega$. On the other hand, the estimates can also be extended up to the boundary, if we assume moreover that $\mf{u}_\beta$ is $L^\infty$ bounded in $\Omega$, $\mf{u}\equiv 0$ on a portion of $\partial \Omega$, and $\pa \Omega$ is there sufficiently smooth.

\begin{theorem}\label{thm: holder bounds boundary}
Under the assumptions of Theorem \ref{thm: holder bounds}, for every $\Omega' \Subset \R^N$, if $\mf{u}_\beta = 0$ on $\Omega' \cap \partial \Omega$ and $\Omega' \cap \partial \Omega$ is smooth, then for any $\alpha\in (0,1)$ there exists $C=C(\Omega',\a)>0$ such that
\[
\| \mf{u}_\b\|_{\Ccal^{0,\alpha}(\Omega' \cap \overline  \Omega)}\leq C.
\]
\end{theorem}

\begin{remark}
A typical example which we have in mind is a system of type \eqref{classical elliptic GP} with competition between groups of components, as in \cite{So}: this means that we consider
\[
-\Delta u_i + \lambda_i u_i = u_i |u_i|^{p-1}\sum_{j=1}^d b_{ij}  |u_j|^{p+1} - \beta u_i |u_i|^{p-1}\sum_{j \neq i} a_{ij} |u_j|^{p+1},
\]
with $b_{ij} \ge 0$ if $(i,j) \in \mathcal{K}_1$ (cooperation inside any group of components) and $b_{ij}=0$ if $(i,j) \in \mathcal{K}_2$ (so that the relation between different groups is described by the second terms on the right hand side, which, as already observed, stays for competition between different groups). It is straightforward to check that with the previous conditions on $b_{ij}$, assumption (H) is satisfied by 
\[
f_{i,\beta}(x,\mf{s}) = s_{i} |s_{i}|^{p-1}\sum_{j=1}^d b_{ij} |s_{j}|^{p+1} -\lambda_i s_{i}.
\]

\end{remark}

From this theorem, we can deduce that, for any such kind of family of solutions $\{\mf{u}_\b\}_\b$, there exists a limiting profile $\mf{u}\in \Ccal_{\rm loc}^{0,\alpha}$ ($\alpha\in (0,1)$) such that, up to a subsequence,
\[
u_{i,\beta}\to u_i \qquad \text{ strongly in } H^1_{\rm loc}\cap \Ccal_{\rm loc}^{0,\alpha}.
\]

We can improve this in the following way, considering also the following assumption for $\mathbf{f}:=\lim_{\beta\to +\infty}\mathbf{f}_\beta$.
\begin{itemize}
\item[(L)] $f_{i}:\Omega\times \R^d \to \R$, and there exists $C>0$ such that
\[
\sup_{i \in I_h} \sup_x  \left|\frac{f_i(x,\mf{s})}{\sum_{j \in I_h} |s_j|} \right|  \le C \qquad \forall \mf{s} \in [0,1]^d,h=1,\ldots, m.
\]
\end{itemize}

\begin{theorem}\label{thm: consequences}
Let $\mf{u}$ be a limiting vector function as before, and assume moreover that $f_{i,\beta}\to f_i$ in $\mathcal{C}_{\loc}(\Omega \times \R^d)$. Then
\begin{enumerate}
\item $\mf{u}_\beta \to \mf{u}$ strongly in $H^1_{\loc}(\Omega)$, and for every compact $K \Subset \Omega$ we have
\[
\beta \int_{K} |u_{i,\beta}|^{p+1} |u_{j,\beta}|^{p+1} \to 0
\]
as $\beta \to \infty$, for every $(i,j) \in \mathcal{K}_2$;
\item for each $h=1,\ldots, m$, and $i\in I_h$, we have
\[
-\Delta u_i=f_{i}(x,\mf{u}) \qquad \text{ in the open set} \left\{\sum_{j\in I_h} u_j^2>0\right\};
\]
\item $u_i u_j\equiv 0$ whenever $(i,j) \in \mathcal{K}_2$ (segregation between groups).
\end{enumerate}
Furthermore, if $\mathbf{f}$ satisfies (L), then
\begin{itemize}
\item[(4)] $u_i$ is Lipschitz continuous in $\Omega$.
\end{itemize}
\end{theorem}

We now turn to the regularity issue in the emerging free boundary problem. For this purpose, we extend Definition \ref{def:old_class_G} to groups of segregated components, each component being possibly sign-changing.

\begin{definition}
We define $\mathcal{G}(\Omega)$ as the set of functions $\mf{u}=(u_1,\dots,u_d) \in H^1(\Omega, \R^d) \setminus \{\mf{0}\}$ such that:
\begin{itemize}
\item[(G1)] $u_i$ are Lipschitz continuous on $\Omega$, and such that $u_i u_j \equiv 0$ in $\Omega$ for every $(i,j) \in \mathcal{K}_2$;
\item[(G2)] each component $u_i$ satisfies
\[
-\Delta u_i = f_i(x,\mf{u}) - \mathcal{M}_i \quad \text{in $\mathcal{D}'(\Omega)= (\mathcal{C}^\infty_c(\Omega))'$},
\]
where $\mathbf{f}$ satisfies (L), and $\mathcal{M}_i$ are nonnegative Radon measures supported on $\Gamma_{\mf{u}}:= \{\mf{u}=\mf{0}\}$.
\item[(G3)] for every $x_0 \in \Omega$ and $0<r<\dist(x_0,\partial \Omega)$ it holds
\begin{align*}
(2-N) \sum_{i=1}^d \int_{B_r(x_0)} |\nabla u_i|^2 &= r \sum_{i=1}^d \int_{\pa B_r(x_0)} \left( 2 (\pa_{\nu} u_i)^2 -|\nabla u_i|^2 \right) \\
& \qquad + 2 \sum_{i=1}^d \int_{B_r(x_0)} f_{i}(x,\mf{u}) \nabla u_i \cdot (x-x_0).
\end{align*}
\end{itemize}
We write that $\mf{u} \in \mathcal{G}_{\loc}(\R^N)$ if $\mf{u} \in \mathcal{G}(B_R(0))$ for every $R>0$.
\end{definition}

Consider the following subset of $\Gamma_{\mf{u}}$:
\[
\widetilde \Gamma_{\mf{u}}=\Omega \setminus \bigcup_{h=1}^m \text{int}\left(\overline{\left\{ \sum_{j\in I_h} u_j^2>0\right\}}\right).
\]
We have the following regularity result.

\begin{theorem}\label{thm:regularity_G}
Let $\mf{u}\in \mathcal{G}(\Omega)$. Then
\begin{itemize}
\item[1.] $\Hh_\text{dim}(\Gamma_\mf{u})\leq N-1$; 
\item[2.] there exists a set $\Rcal_\mf{u}\subseteq \widetilde\Gamma_\mf{u}$, relatively open in $\widetilde \Gamma_\mf{u}$, such that
\begin{itemize}
\item[-] $\Hh_\text{dim} (\widetilde \Gamma_\mf{u}\setminus \Rcal_\mf{u})\leq N-2$;
\item[-] $\Rcal_\mf{u}$ is a collection of  hypersurfaces of class $\Ccal^{1,\alpha}$ (for some $0<\alpha<1$), each one locally separating two connected components of $\Omega\setminus \Gamma_{\mf{u}}$.
\item[-] given $x_0\in \Rcal_\mf{u}$, there exist $h,k\in \{1,\ldots, m\}$ such that
\begin{equation}\label{eq:reflectionlaw}
\lim_{x\to x_0^+} \sum_{i\in I_h} |\nabla u_i|^2=\lim_{x\to x_0^-} \sum_{i\in I_k} |\nabla u_i|^2\neq 0,
\end{equation}
where $x\to x_0^\pm$ are limits taken from opposite sides of the hypersurface.
\item[-] whenever $x\in\widetilde \Gamma_\mf{u}\setminus \Rcal_\mf{u}$, we have
\begin{equation}\label{eq:vanishingofgradient}
\sum_{i=1}^d |\nabla u_i(x)|^2\to 0 \qquad \text{ as } x\to x_0.
\end{equation}
\end{itemize}
\item[3.] Furthermore, if $N=2$, then $\Rcal_{\mf{u}}$ consists in a locally finite collection of curves meeting with equal angles at singular points.
\end{itemize}
If $u\in \mathcal{G}(\Omega)$ is such that $u_i\geq 0$ for every $i$, then conclusions 1.-3. hold with $\Gamma_{\mf u}$ instead of $\widetilde \Gamma_{\mf{u}}$
\end{theorem}

We remark that having sign-changing solutions  adds some difficulties to the proof of the previous theorem, since one needs to take into account the intersection of the nodal set of each individual component with the common nodal set of all components.  However, during the proof we will show that in the neighbourhood of each regular point of $\widetilde \Gamma_{\mf{u}}$ there are always components which do not change sign. For elements in $\mathcal{G}(\Omega)$ with sign-changing components, we need to deal with $\widetilde \Gamma_{\mf{u}}$. This is due to the fact that, in general, we cannot exclude the existence of points $x_0\in \Gamma_{\mf{u}}$ for which there exists a small $\delta>0$ such that $B_\delta(x_0)\setminus \Gamma_{\mf{u}}$ is a connected set. In some particular situations, such as in the framework of \cite{rtt}, these points can be excluded (see Corollary 3.24 in \cite{rtt}); in general, for elements of $\mathcal{G}(\Omega)$ with nonnegative components, this can be always excluded.

\begin{theorem}\label{thm: limit behavior}
Under the assumptions of Theorem \ref{thm: holder bounds}, suppose furthermore that $f_{i,\beta}\to f_i$ in $\mathcal{C}_{\loc}(\Omega \times \R^d)$ with $\mathbf{f}$ satisfying (L), and that the limiting profile (as $\beta\to \infty$) $\mf{u}$ is such that $u_i \not \equiv 0$ in $\Omega$ for at least some $i$. Then $\mf{u}\in\Gcal(\Omega)$. In particular, the limiting profile satisfies all the conclusions of Theorem \ref{thm:regularity_G}.
\end{theorem}

To conclude, we observe that a couple of problems addressed and solved for family of solutions to \eqref{classical elliptic GP} remains open in our general context: firstly, the proof of the uniform boundedness in the Lipschitz space, as in \cite{SoZi}; secondly, the precise description of the singular set in the emerging free boundary problem, as in \cite{CaffLin2010,ZhangLiu}. These will be object of future investigation.

\subsection{Structure of the paper}

The paper is organized as follows: in Section \ref{sec: bounds} we prove Theorems \ref{thm: holder bounds} and \ref{thm: holder bounds boundary}. We follow the structure of the proof of Theorem 1.1 in \cite{nttv}, but, as we shall see, we have to face several complications which mainly arise from the fact that we have a non-trivial grouping among the different components, and that we deal with arbitrary exponents $p>0$ (thus including sublinear terms). Section \ref{sec: properties} is devoted to the proof of Theorem \ref{thm: consequences}. In Section \ref{sec: regularity} we present the proofs of Theorems \ref{thm:regularity_G} and \ref{thm: limit behavior}. This part differs substantially with respect to \cite{CaffLin,tt}, since, as we shall see, the effect of the nontrivial grouping together with the fact that we do not consider minimal solutions introduce several complications.  In particular, a new boundary Harnack Principle is proved in Subsection \ref{subseq:Harnack}. Finally, we collect all the Liouville-type theorems that we used in the paper in an appendix, for the reader's convenience; although most of such results are already known, we need also new ones to treat the case $0<p<1$.

\section{Proof of the uniform H\"older bounds}\label{sec: bounds}

In this section we prove first Theorem \ref{thm: holder bounds}, and will assume from now on its assumptions. The proof closely follows those of Theorem 1.1 in \cite{nttv} and of Theorem 2.6 in \cite{Wa} (see also \cite[Theorem 3.11]{rtt}), with the necessary modifications which come from the fact that we are considering a ``non purely competitive" setting, sign-changing solutions, and interactions with general $p >0$ (in case smaller than 1). Without loss of generality we suppose that $\Omega \supset B_3$, and we aim at proving the uniform H\"older bound in $B_1$. We know that 
\[
\sup_{i=1,\dots,d} \|u_{i,\beta}\|_{L^\infty(B_3)} \le M <+\infty
\]
independently on $\beta$. Let $\eta \in \mathcal{C}^1_c(\R^N)$ be a radially decreasing cut-off function such that 

\begin{equation}\label{def eta}
\begin{cases}
\eta(x) =1 & \text{for $x \in B_1$}\\
\eta(x)= 0 & \text{for $x \in \R^N \setminus B_2$} \\
\eta(x) = (2-|x|)^2 & \text{for $x \in B_2 \setminus B_{3/2}$}.
\end{cases}
\end{equation}
The explicit shape of $\eta$ in $B_2\setminus B_{3/2}$ will allow us to control the ratio $\eta(x)/\eta(y)$ for $x,y$ in certain balls that are close to $\partial B_2$, see Remark \ref{rmk: bdd ratio cut off} ahead. We aim at proving that the family $\{\eta \mf{u}_\beta: \beta >0\}$ admits a uniform bound on the $\alpha$-H\"older semi-norm, that is, there exists $C > 0$, independent of $\beta$, such that
\begin{equation}\label{lip scaling}
\sup_{i=1,\dots,d} \sup_{\substack{x \neq y \\ x,y \in \overline{B_2}}}    \frac{ |(\eta u_{i,\beta})(x)-(\eta u_{i,\beta})(y)|}{|x-y|^{\alpha}} \leq C.
\end{equation}
Since $\eta=1$ in $B_1$, once \eqref{lip scaling} is proved, Theorem \ref{thm: holder bounds} follows. 

If $\beta$ varies in a bounded interval, then such a uniform bound does exist by elliptic regularity. Indeed, in such a case, since both $f_{i,\beta}$ and $u_{i,\beta}$ are uniformly bounded in $L^\infty(B_2)$, also
\[
f_{i,\beta}(x,\mf{u}_\beta)-\beta \mathop{\sum_{j=1}^d}_{j\neq i} a_{ij} u_{i,\beta} |u_{i,\beta}|^{p-1}|u_{j,\beta}|^{p+1} \qquad \text{ is uniformly bounded in } B_2.
\]
Thus, we may conclude using the classical estimate \cite[Theorem 9.11]{GiTr} and the embeddings \cite[Theorem 7.26]{GiTr}.
Hence, let us assume by contradiction that there exists a sequence $\beta_n \to +\infty$ and a corresponding sequence $\{\mf{u}_n\}$ such that
\begin{equation}\label{absurd assumption}
    L_n := \sup_{i = 1, \dots, d} \sup_{\substack{x \neq y \\ x,y \in \overline{B_2}} } \frac{ |(\eta u_{i,n})(x)-(\eta u_{i,n})(y)|}{|x-y|^{\alpha}} \to \infty \qquad \text{as $n \to +\infty$.}
\end{equation}
Up to a relabelling, we may assume that the supremum is achieved for $i = 1$ and at a pair of points $x_n, y_n \in \overline{B_{2}}$ and moreover, $x_n \neq y_n$ since, for $\beta_n$ fixed, the functions $\mf{u}_{i,n}$ are smooth. As $\{\mf{u}_\beta\}$ is uniformly bounded in $L^\infty(B_2)$, it is immediate to observe that $|x_n-y_n| \to 0$ as $n \to \infty$, since
\[
|x_n-y_n|^{\alpha}= \frac{|(\eta u_{1,n})(x)-(\eta u_{1,n})(y)|}{L_n} \le \frac{C}{L_n}.
\]

\subsection{Blow-up analysis}

As in \cite{SoZi, tvz1, Wa} the contradiction argument is based on two blow-up sequences:
\begin{equation}\label{def blow-up}
    v_{i,n}(x) := \eta(x_n) \frac{u_{i,n}(x_n + r_n x)}{L_n r_n^{\alpha}} \quad \text{and} \quad \bar{v}_{i,n}(x) := \frac{(\eta u_{i,n})(x_n + r_n x)}{L_n r_n^{\alpha}},
\end{equation}
both defined on the scaled domain $(\Omega - x_n)/r_n \supset (B_3-x_n)/r_n=: \Omega_n$. The function $\bar{\mf{v}}_{n}$ is the one for which the H\"older quotient is normalized (see Lemma \ref{lem: basic prop}-(1) ahead), however it satisfies a rather complicated system. On the other hand, its localized version $\mf{v}_{n}$, as we will see, satisfies a simple system related to \eqref{eq:main_system}. We will also check that both blow-up functions have (locally) comparable $L^\infty$ norms and gradients (as a byproduct of Remark \ref{rmk: bdd ratio cut off} below), and this allows to interchange information from one function to the other. This idea goes back to the ``freezing of the coefficients'' used in the proof of the classical Schauder estimates (see for instance Section 6 in \cite{GiTr}), and was firstly used in this context by K. Wang \cite{Wa}.

The functions $\bar{\mf{v}}_n$ are non-trivial in the subset $(B_{2}-x_n)/r_n=:\Omega_n'$. Here $0<r_n\to 0$ will be conveniently chosen later. Note that 
$\{\Omega_n'\}$ converges to a limit domain $\Omega_\infty$, which can be a half-space or the entire space according to the asymptotic behaviour of the sequence
\[
\dist(0,\pa \Omega_n')= \frac{\dist(x_n,\pa B_2)}{r_n}.
\] 
On the other hand, since $\Omega_n \supset B_{1/r_n}$, in the limit as $n \to \infty$ it results that $\Omega_n$ approaches $\R^N$.
The following remark, that originates from the explicit definition of $\eta$ in $B_2\setminus B_{3/2}$, will allow us to compare the gradients of $v_{i,n}$ and $\bar v_{i,n}$, which will be essential in the proofs of Lemma \ref{lem: bound in 0} and Lemma \ref{M_n illimitato}. 
\begin{remark}\label{rmk: bdd ratio cut off}
For an arbitrary $x \in B_2$, let $r_x:= |x|$ and $d_x:= \dist(x,\pa B_2) = 2-r_x$. In light of \eqref{def eta}, it is possible to check that 
\[
\sup_{x \in B_2} \sup_{\rho \in (0,d_x/2)}  \frac{\sup_{B_{\rho}(x)} \eta    }{\inf_{B_\rho(x)} \eta} \le 16.  
\]
Indeed, for any $x \in B_2 \setminus B_{7/4}$ and for every $\rho \in (0,d_x/2)$, we have $B_{d_x/2}(x)\subset B_2\setminus B_{3/2}$, and
\[
\sup_{B_{\rho}(x)} \eta \le \sup_{B_{d_x/2}(x)} \eta = \left(2-r_x+\frac{d_x}{2}\right)^2 = \frac{9}{4}(2-r_x)^2,
\]
and
\[
\inf_{B_{\rho}(x)} \eta \ge \inf_{B_{d_x/2}(x)} \eta = \left(2-r_x-\frac{d_x}{2}\right)^2 = \frac{1}{4}(2-r_x)^2,
\]
On the other hand, for $x\in B_{3/2}$, we have $B_{d_x/2}(x)\subset  B_{7/4}$, and
\[
\sup_{B_{d_x/2}(x)} \eta \leq 1,\qquad \text{ and } \qquad \inf_{B_{d_x/2}(x)} \eta\geq \inf_{B_{7/4}(0)} \eta\geq \left(2-\frac{7}{4}\right)^2=\frac{1}{16}.
\]
\end{remark}

Basic properties of the blow-up sequences are collected in the following lemma.

\begin{lemma}\label{lem: basic prop}
In the previous setting, it results that:
\begin{enumerate}
\item the sequence $\{\bar{\mf{v}}_n\}$ has uniformly bounded $\alpha$-H\"older semi-norm in $\Omega'_n$, and in particular
\[
\sup_{i=1,\dots,d} \sup_{\substack{x \neq y \\ x,y \in \overline{\Omega_n'}}}    \frac{ |\bar v_{i,n}(x)-\bar v_{i,n}(y)|}{|x-y|^{\alpha}} = \frac{ |\bar v_{1,n}(0)-\bar v_{1,n}\left(\frac{y_n-x_n}{r_n}\right)|}{\left|\frac{y_n-x_n}{r_n}\right|^{\alpha}} = 1
\]
for every $n$.
\item $v_{i,n}$ is a solution of 
\begin{equation}\label{system blow-up}
-\Delta v_{i,n} = g_{i,n}(x) - M_n v_{i,n} |v_{i,n}|^{p-1}  \sum_{j \neq i} a_{ij} |v_{j,n}|^{p+1} \qquad \text{in $\Omega_n$},
\end{equation}
where
\[
 M_n:= \beta_n r_n^{2(\alpha p +1)} \left(\frac{L_n}{\eta(x_n)}\right)^{2p}.
\] 
and
\[
\begin{cases}
\displaystyle  g_{i,n}(x):= \frac{\eta(x_n) r_n^{2-\alpha}  }{L_n} f_{i,\b_n}(x_n+r_n x) &  \text{ if } p\geq 1\\[10pt]
\displaystyle g_{i,n}(x):= \frac{\eta(x_n) r_n^{2-\alpha}  }{L_n} f_{i,\b_n}(x_n+r_n x,\mf{u}_{n}(x_n+r_n x)) &\text{ if } 0<p<1.
\end{cases}
\]
\item $\|g_{i,n}\|_{L^\infty(\Omega_n)}\to 0$ as $n \to \infty$. Moreover, if $0<p<1$
\[
|g_{i,n}(x)| \leq \textrm{o}_n(1) \sum_{j\in I_h} |v_{j,n}|^p\quad   \text{for every $i\in I_h$}.
\]
\item for every compact set $K \subset \R^N$ we have 
\[
\sup_K |\mf{v}_n - \bar{\mf{v}}_n | \to 0 \qquad \text{as $n \to \infty$}.
\]
\item for every compact $K \subset \R^N$ there exists $C>0$ such that 
\[
|v_{i,n}(x) - v_{i,n}(y)| \le C + |x-y|^\alpha
\]
for every $x, y \in K$ and $i=1,\dots,d$; in particular $\{v_{i,n}\}$ has uniformly bounded oscillation in any compact set.
\end{enumerate}
\end{lemma}
\begin{proof}
The proof of points (1)-(2) is trivial. For (3), it is sufficient to use the definition of $g_{i,n}$ and the boundedness of $\{\mf{u}_n\}$ in $L^\infty(\Omega)$, plus assumption (H). As far as (4) is concerned, since $\eta$ is globally Lipschitz continuous with constant denoted by $l$, and $\{u_{i,n}\}$ is uniformly bounded in $K$, we have 
\[
|v_{i,n}(x) - \bar v_{i,n}(x)|= \frac{|u_{i,n}(x_n+r_n x)|}{L_n r_n^{\alpha}} |\eta(x_n)-\eta (x_n+r_n x)| \le \frac{l M r_n^{1-\alpha}}{L_n} |x|,
\]
where we recall that $\|u_{i,n}\|_{L^\infty(B_3)} \le M$ for every $i$ and $n$. Finally, for (5) we use point (4) and the uniform H\"older boundedness of the sequence $\{\bar{\mf{v}}_n\}$.
\end{proof}



\begin{lemma}\label{lem: bound in 0}
Take $0<r_n \to 0$ such that
\begin{equation}\label{conditions r_n}
\liminf_n M_n>0, \qquad \limsup_n \frac{|x_n-y_n|}{r_n}<\infty.
\end{equation}
Then the sequence $(\mf{v}_n(0))$ is bounded. 
\end{lemma}

\begin{remark}
Although the statement is the same as Lemma 3.4 in \cite{nttv}, due to the different assumptions the proof is very different and thus we shall present it in detail. 
\end{remark}
\begin{proof}
Take $R$ such that $R\geq |y_n-x_n| /r_n$ for every $n$, and assume by contradiction that $|\mf{v}_n(0)|\to +\infty$. Since $\mf{v}_n(0)=\bar{\mf{v}}_n(0)$, and $\{\bar{\mf{v}}_n\}$ has uniformly bounded $\alpha$--H\"older semi-norm (recall Lemma \ref{lem: basic prop}-(1)) we have
\[
|\bar{\bf{v}}_n(0)| \leq \inf_{B_{2R}}|\bar{\mf{v}}_n|+(4R)^{\alpha},\quad \text{ hence} \quad  \inf_{B_{2R}}  | \bar{\mf{v}}_n|\to \infty.
\]
Since $\bar{\mf{v}}_n|_{\partial \Omega_n'}\equiv 0$, we have $B_{2R}\subset \Omega_n'$ for sufficiently large $n$. We observe moreover that, since we can take $R$ arbitrary large, this means that, in the present setting, $\Omega_n'$ exhausts $\R^N$ as $n\to \infty$, and so necessarily
\begin{equation}\label{dist to infty}
\frac{\text{dist}(x_n,\partial B_2)}{r_n}\to +\infty.
\end{equation}

Let $\varphi \in \mathcal{C}^\infty_c(B_{2R})$ be a nonnegative function such that $\varphi=1$ in $B_R$. Fix $h\in \{1,\ldots, m\}$ and take $i\in I_h$.  By testing the equation for $v_{i,n}$ in \eqref{system blow-up} against $v_{i,n} \varphi^2$, we obtain (recall that $a_{ij}=0$ for $j\in I_h$)
\begin{multline}
\int_{B_{2R}}|\nabla v_{i,n}|^2 \varphi^2 + M_n \int_{B_{2R}} |v_{i,n}|^{p+1}\sum_{j\not\in I_h} a_{ij}|v_{j,n}|^{p+1}\varphi^2  \\
= -\int_{B_{2R}} 2 v_{i,n}\varphi \nabla v_{i,n}\cdot \nabla \varphi + \int_{B_{2R}} g_{i,n} v_{i,n}\varphi^2 
\leq \frac{1}{2}\int_{B_{2R}} |\nabla v_{i,n}|^2\varphi^2 + C\int_{B_{2R}}(v_{i,n}^2 + 1),
\end{multline}
where in the last equality we used point (3) of Lemma \ref{lem: basic prop}. Summing up for $i\in I_h$, we have, whenever $k\neq h$,
\begin{equation}\label{eq:auxiliary_v_n(0)_bounded1}
M_n\int_{B_R} \sum_{i\in I_h} |v_{i,n}|^{p+1} \sum_{j\in I_k} |v_{j,n}|^{p+1} \leq C \int_{B_{2R}} \sum_{i\in I_h} (|v_{i,n}|^2+1);
\end{equation}
hence, by using at first the fact that $\liminf M_n>0$, and afterwards the boundedness of the oscillation of $\{v_{i,n}\}$ (see Lemma \ref{lem: basic prop}-(5)), we deduce that for every $x$ in $B_R$
\[
\left(\sum_{i\in I_h} |v_{i,n}(x)| \right)^{p+1} \left(\sum_{j\in I_k} |v_{j,n}(x)|\right)^{p+1} \leq C_1 \left(\sum_{i\in I_h} |v_{i,n}(x)|\right)^2 +C_2
\]
In particular, for every $k\neq h$, $x\in B_R$, it results that
\begin{multline}\label{eq:comparison_between_groups}
\left(\sum_{i\in I_h} |v_{i,n}(x)| \right)^{2(p+1)} \left(\sum_{j\in I_k} |v_{j,n}(x)|\right)^{2(p+1)} \\
\leq C\left( \left(\sum_{i\in I_h} |v_{i,n}(x)|\right)^2 +1\right)\left(\left(\sum_{i\in I_k} |v_{i,n}(x)|\right)^2 +1\right),
\end{multline}
where $C>0$ depends only on $R$. Evaluating this inequality at $x=0$, and since $|\mf{v}_n(0)|\to +\infty$ and $p>0$, there exists exactly one $\bar h$ such that
\[
\sum_{i\in I_{\bar h}} |v_{i,n}(0)|\to +\infty,\quad \text{ whereas} \quad \sum_{j\in I_{k}} |v_{j,n}(0)|\text{ is bounded },\ \forall k\neq \bar h.
\]
This implies, once again by Lemma \ref{lem: basic prop}, that
\[
\inf_{B_{2R}}\sum_{i\in I_{\bar h}} |v_{i,n}|\to +\infty,\qquad \sup_{B_{2R}} \sum_{j\in I_k} |v_{j,n}| \text{ is bounded }, \forall k\neq \bar h,
\]
and from \eqref{eq:comparison_between_groups} we have that actually
\begin{equation}\label{eq:comparison_between_groups2}
\sup_{B_{2R}} \sum_{j\in I_k} |v_{j,n}| \to 0\qquad \forall k\neq \bar h.
\end{equation}

We now split the proof in two cases, and four subcases:
\medbreak

\noindent \emph{Case 1.} $p\geq 1$. 
\smallbreak

\noindent \emph{Subcase 1.1. }  $\bar h=1$, the index associated to the group with the non-constant function $\bar v_{1,n}$.
In this situation, let
\[
I_n:= M_n \inf_{B_{2R}} \sum_{i\in I_1}|v_{i,n}|^{p+1} \to +\infty,
\]
We recall also that $\sup_{B_{2R}} |v_{j,n}|\to 0$ for every $j\not\in I_1$; for such $j$'s, by the Kato inequality (see e.g. \cite{Br84})
\begin{align*}
-\Delta |v_{j,n}| &\le |g_{j,n}(x)| - M_n  |v_{j,n}|^{p} \sum_{k \in I_1} a_{jk} |v_{k,n}|^{p+1} \\
& \le \|g_{j,n}\|_{L^\infty(B_R)} - \kappa I_n |v_{j,n}|^p \qquad \text{ in } B_{2R}.
\end{align*}
Thus by the decay estimate \cite[Lemma 2.2]{SoZi} we have
\[
I_n \sup_{B_R} |v_{j,n}|^p \leq \frac{C}{R^2} \sup_{B_R} |v_{j,n}| + \sup_{B_R}|g_{j,n}|= \textrm{o}_n(1) \qquad \forall j\notin I_1.
\]
In particular, for $x\in B_R$,
\[
M_n |v_{1,n}|^p \sum_{j\not\in I_1}|v_{j,n}|^{p+1}\leq \textrm{o}_n(1)\frac{\sup_{B_{R}} \sum_{i\in I_1}|v_{i,n}|^{p} }{M_n^{1/p} \inf_{B_{2R}} \sum_{i\in I_1}|v_{i,n}|^{(p+1)^2/p}} \to 0,
\]
as $(p+1)^2/p>p$ (use also Lemma \ref{lem: basic prop}-(5)).
Thus, as $\|g_{i,n}\|_{L^\infty(B_R)} \to 0$, we have
\begin{equation}\label{eq:Laplac_goes_to_0}
\|\Delta v_{i,n}\|_{L^\infty(B_R)}\to 0
\end{equation}
for every sufficiently large $R>0$. We can now conclude this case adapting some ideas from \cite[p.281-292]{nttv}; here the situation is more delicate, because we need to take in account the presence of the function $\eta$. Take the new sequence $w_n(x):= v_{1,n}(x)- v_{1,n}(0)$. Then Lemma \ref{lem: basic prop}-(5) and \eqref{eq:Laplac_goes_to_0} combined with the Ascoli-Arzela theorem yields that $w_{n}\to w_\infty$ in $L^\infty_{\rm loc}$, where $w_\infty$ is a harmonic function defined in $\R^N$ (recall that the contradiction assumption implies that $\Omega_n'$ approaches $\R^N$). We claim that 
\[
\mathop{\max_{x,y\in \overline \R^N}}_{x\neq y} \frac{|w_\infty(x)- w_\infty(y)|}{|x-y|^\alpha}=1.
\]
If this holds, we immediately have a contradiction with Lemma \ref{lem:liouville1} in the appendix. In order to prove the claim, we need to consider the blow-up sequence $\{\bar v_n\}$, and the auxiliary function $\tilde w_n(x)=\bar v_{1,n}(x)-\bar v_{1,n}(0)$. From Lemma \ref{lem: basic prop}-(4), we have that also $\tilde w_n\to w_\infty$ in $L^\infty_{\rm loc}$. Thus, since we have Lemma \ref{lem: basic prop}-(1), we are left to prove that $\liminf |y_n-x_n|/r_n>0$. Let $z_\infty$ be the limit of any convergent subsequence. From \eqref{eq:Laplac_goes_to_0}, we have that $\{w_n\}$ is uniformly bounded in $\mathcal{C}^{1,\gamma}(\overline B_R)$, for every $0<\gamma<1$. We claim that also $\{|\nabla \bar v_{1,n}|\}$ is bounded in $L^\infty(B_R)$, and to prove our claim we observe that, since by definition
\[
\bar v_{1,n}(x) = \frac{\eta(x_n+r_n x)}{\eta(x_n)} v_{1,n}(x),
\] 
we have
\begin{equation}\label{eq gradienti}
\begin{split}
\nabla \bar v_{1,n}(x) &=  \frac{\eta(x_n+r_n x)}{\eta(x_n)}\nabla v_{1,n}(x) + \frac{r_n v_{1,n}(x)}{\eta(x_n)} \nabla \eta(x_n+r_n x) \\
& =   \frac{\eta(x_n+r_n x)}{\eta(x_n)}\nabla v_{1,n}(x) + \frac{r_n^{1-\alpha}u_{1,n}(x_n+r_n x) }{L_n} \nabla \eta(x_n+r_n x) \\
& = \frac{\eta(x_n+r_n x)}{\eta(x_n)}\nabla v_{1,n}(x) + O\left(\frac{r_n^{1-\alpha} }{L_n}\right),
\end{split}
\end{equation}
where we used the uniform $L^\infty$-boundedness of the sequence $\{\mf{u}_n\}$. Let $K$ be a compact set of $\R^N$. By \eqref{dist to infty} we have
\[
\sup_{x \in K} |x_n+ r_n x-x_n|  = r_n C(K) \le \frac{\dist(x_n,\pa B_2)}{2}
\]
for every $n$ sufficiently large, so that 
\[
\sup_{x \in K} \frac{\eta(x_n+r_n x)}{\eta(x_n)} \le  \sup_{x \in B_2} \sup_{\rho \in (0,d_x/2)} \frac{\sup_{B_{\rho}(x)} \eta}{\inf_{B_{\rho}(x)} \eta} \le C,
\]
see Remark \ref{rmk: bdd ratio cut off}. As a consequence
\[
\sup_K |\nabla \bar v_{1,n}| \le C \sup_K |\nabla v_{1,n}| + O(r_n^{1-\alpha} L_n^{-1}) \le C
\]
as $n \to \infty$, that is, the sequence $\{|\nabla \bar v_{1,n}|\}$ is locally uniformly bounded.

Now, if $|y_n-x_n|/r_n \to 0$ we would have
\begin{align*}
1 &= \frac{ \left|\bar v_{1,n}(0)-\bar v_{1,n}\left(\frac{y_n-x_n}{r_n}\right)\right|}{\left|\frac{y_n-x_n}{r_n}\right|^{\alpha}} =    \frac{ \left|\bar v_{1,n}(0)-\bar v_{1,n}\left(\frac{y_n-x_n}{r_n}\right)\right|}{\left|\frac{y_n-x_n}{r_n}\right|} \left|\frac{y_n-x_n}{r_n}\right|^{1-\alpha} \le C \left|\frac{y_n-x_n}{r_n}\right|^{1-\alpha} \to 0,
\end{align*}
a contradiction. Thus, $z_\infty \neq 0$, which completes the proof of this case.

\smallbreak

\noindent \emph{Subcase 1.2. } $\bar h>1$, so that there is a non-constant function $\bar v_{1,n}$ which is not in the group $I_{\bar h}$.
In this case, let
\[
I_n:= M_n \inf_{B_{2R}} \sum_{i\not\in I_1}|v_{i,n}|^{p+1} \to +\infty,
\]
and recall that $\sup_{B_{2R}} |v_{1,n}|\to 0$.
Therefore by the Kato inequality
\begin{align*}
-\Delta |v_{1,n}| &\le |g_{1,n}(x)| - M_n  |v_{1,n}|^{p}  \sum_{k \not\in I_1} a_{1k} |v_{k,n}|^{p+1} \\
& \le \|g_{1,n}\|_{L^\infty(B_R)} - a_{1j} I_n |v_{1,n}|^p \qquad \text{ in } B_{2R}.
\end{align*}
Once again by the decay estimate \cite[Lemma 2.2]{SoZi} we have
\[
I_n \sup_{B_R} |v_{1,n}|^p =\textrm{o}_n(1).
\]
In particular, for $x\in B_R$,
\[
M_n |v_{1,n}|^p \sum_{i\not\in I_1}|v_{j,n}|^{p+1}\leq \textrm{o}_n(1)\frac{\sup_{B_R} \sum_{i\not\in I_1}|v_{i,n}|^{p+1} }{\inf_{B_{2R}} \sum_{i\not\in I_1}|v_{i,n}|^{p+1}} \to 0.
\]
Thus we obtain once again \eqref{eq:Laplac_goes_to_0}, and get a contradiction as before.

\medbreak

\noindent \emph{Case 2.} $0<p<1$.

\smallbreak 

\noindent \emph{Subcase 2.1} $\bar h=1$. Take once again
\[
I_{n}:= M_n \inf_{B_{2R}} \sum_{i\in I_1} |v_{i,n}|^{p+1} \to +\infty,
\] 
and recall that, for $k\neq 1$, $\sup_{B_{2R}} \sum_{j\in I_k}|v_{j,n}|\to 0$, so in particular  $\sum_{j\in I_k}|v_{j,n}|^p\geq  \sum_{j\in I_h}|v_{j,n}|$ in $B_{2R}$ for large $n$. So, for every $j\in I_k$, recalling Lemma \ref{lem: basic prop}-(3),
\[
-\Delta |v_{j,n}| \le C \sum_{j\in I_k}|v_{j,n}|^p - \kappa I_{n} |v_{j,n}|^p \qquad \text{ in } B_{2R}.
\]
Summing up in $I_k$,
\[
-\Delta \left(\sum_{j\in I_k} |v_{j,n}|\right) \le C' \sum_{j\in I_k}|v_{j,n}|^p - \kappa I_{n} \sum_{j\in I_k}|v_{j,n}|^p\leq -\tilde C I_{n} \sum_{j\in I_k} |v_{j,n}|^p\leq -\tilde C I_{n} \sum_{j\in I_k} |v_{k,n}|
\]
Thus, by the decay estimate \cite[Lemma 4.4]{ctv}, we have
\[
\sup_{B_R} \sum_{j\in I_k} |v_{j,n}|\leq C_1 e^{-C_2\sqrt{I_{n}}} \qquad \forall k\neq 1
\]
and so
\[
|\Delta v_{1,n}|\leq |g_{1,n}|+ M_n |v_{1,n}|^p  \sum_{j\not \in I_1} |v_{j,n}|^{p+1} \leq \|g_{1,n}\|_{L^\infty(B_R)}+ 2I_n C_1 e^{-(p+1)C_2\sqrt{I_{n}}}\to 0
\]
uniformly in $B_R$. This implies \eqref{eq:Laplac_goes_to_0}, which leads to a contradiction.

\smallbreak 

\noindent \emph{Subcase 2.2} $\bar h>1$. In this final case, reasoning as before,
\[
|v_{1,n}|\leq \sum_{i\in I_1} |v_{i,n}|\leq C_1 e^{-C_2\sqrt{I_{n}}}
\]
where this time 
\[
I_n:= M_n \inf_{B_{2R}} \sum_{i\not\in I_1}|v_{i,n}|^{p+1} \to +\infty,
\]
and again \eqref{eq:Laplac_goes_to_0} holds, as
\[
|\Delta v_{1,n}|\leq |g_{1,n}|+ M_n |v_{1,n}|^p  \sum_{j\not \in I_1} |v_{j,n}|^{p+1} \leq \|g_{1,n}\|_{L^\infty(B_R)}+ 2I_n C_1 e^{-pC_2\sqrt{I_{n}}}\to 0. \qedhere
\]
\end{proof}

\begin{lemma}\label{M_n illimitato}
Up to a subsequence it results that
\[
 \beta_n \left( \frac{L_n}{\eta(x_n)}\right)^{2p}   |x_n-y_n|^{2(\alpha p +1)} \to +\infty
 \]
 as $n \to \infty$.
\end{lemma}
\begin{proof}
By contradiction, let us assume that the sequence of the thesis is bounded. We choose
\[
r_n:=  \left( \beta_n \left( \frac{L_n}{\eta(x_n)}\right)^{2p} \right)^{-1/(2(\alpha p +1))},
\]
so that $M_n=1$ for every $n$. Since the condition \eqref{conditions r_n} is satisfied, we can apply Lemma \ref{lem: bound in 0} and conclude that the sequence $\{\bar{\mf{v}}_n\}$ is bounded at $0$. Thus, by uniform H\"older continuity, it converges uniformly on compact sets of $\Omega_\infty$ to a globally $\alpha$-H\"older continuous function $\mf{v}$. Furthermore, since $(M_n)$ is bounded and $v_{i,n}$ is defined in $\Omega_n$, the fact that $\mf{v}_n$ solves system \eqref{system blow-up} implies that $\{\mf{v}_n\}$ is locally bounded in $\mathcal{C}^{1,\alpha}$. In particular, for every $R>0$ there exists $C>0$ such that
\[
\sup_{i=1,\dots,d} \sup_{B_R} |\nabla v_{i,n}| \le C.
\]

In case $\Omega_\infty=\R^N$ (which happens if $\dist(x_n,\pa B_2)/r_n \to +\infty$), arguing as in the proof of Lemma \ref{lem: bound in 0} it is possible to show that moreover $v_1$ is not constant. Without loss of generality, we assume that $v_1^+$ is not constant. By uniform convergence and thanks to point (2) in Lemma \ref{lem: basic prop}, we have that 
\begin{equation}\label{system entire}
-\Delta v_i = -|v_{i}|^{p-1} v_i \sum_{j \neq i} a_{ij}  |v_j|^{p+1}   \qquad \text{in $\R^N$}
\end{equation}
for every $i=1,\dots,d$. In particular
\[
\begin{cases}
-\Delta v_1^+  \le - a_{1j} (v_1^+)^{p}   (v_j^+)^{p+1} \\
-\Delta v_j^+  \le -a_{1j} (v_j^+)^{p}  (v_1^+)^{p+1}
\end{cases} \quad \text{and} \quad 
\begin{cases}
-\Delta v_1^+  \le -a_{1j} (v_1^+)^{p}  (v_j^-)^{p+1} \\
-\Delta v_j^- \le -a_{1j} (v_j^-)^{p}  (v_1^+)^{p+1}
\end{cases} 
\]
for every $j \not\in I_1$. By global H\"older continuity we are in position to apply Lemma \ref{lem:liouville3}, deducing that $v_j \equiv 0$ for every $j \not\in I_1$. But then $v_1$ is a harmonic H\"older continuous non-constant function, a contradiction.

In case $\Omega_\infty$ is a half-space, then necessarily the sequence $(\dist(x_n,\pa \Omega_n)/r_n)$ is bounded. In such a situation, let us prove first that $|y_n-x_n|/r_n \not \to 0$. If $z_n:= (y_n-x_n)/r_n \to 0$, then
\begin{align*}
|z_n| &\le |z_n|^{\alpha} = |\bar v_{1,n}(0)- \bar v_{1,n}(z_n)| \le \frac{2 m}{L_n r_n^{\alpha}} (\eta(x_n)+\eta(y_n)) \\
& \le  \frac{2 l m }{L_n r_n^{\alpha}} \left( \dist(x_n,\pa B_2) + \dist(y_n,\pa B_2) \right) = \frac{2 l m r_n^{1-\alpha}}{L_n } \left( 2\frac{\dist(x_n,\pa B_2)}{r_n} + \frac{|y_n-x_n|}{r_n} \right),  
\end{align*}
which implies that 
\[
|z_n| \le \frac{\dist(x_n,\pa B_2)}{2r_n}= \frac{\dist(0,\pa \Omega_n')}{2}
\]
for every $n$ sufficiently large. By Remark \ref{rmk: bdd ratio cut off} and the estimate \eqref{eq gradienti}, it results that
\begin{align*}
\sup_{B_{\dist(0,\pa \Omega_n')/2}}|\nabla \bar v_{1,n}| &\le \sup_{y \in  B_{\dist(x_n,\pa B_2)/(2r_n)}} \frac{\eta(x_n+r_n y)}{\eta(x_n)} |\nabla v_{1,n}|+ o_n(1) \\
& \le \sup_{x \in B_2} \sup_{\rho \in (0,d_x/2)} \frac{\sup_{B_{\rho}(x)}\eta}{\inf_{B_\rho(x)} \eta}  + o_n(1)\le C.
\end{align*}
Therefore
\begin{align*}
1 &= \frac{ \left|\bar v_{1,n}(0)-\bar v_{1,n}\left(\frac{y_n-x_n}{r_n}\right)\right|}{\left|\frac{y_n-x_n}{r_n}\right|^{\alpha}} =    \frac{ \left|\bar v_{1,n}(0)-\bar v_{1,n}\left(\frac{y_n-x_n}{r_n}\right)\right|}{\left|\frac{y_n-x_n}{r_n}\right|} \left|\frac{y_n-x_n}{r_n}\right|^{1-\alpha} \\
& \le \sup_{B_{\dist(0,\pa \Omega_n')/2}}|\nabla \bar v_{1,n}| \left|\frac{y_n-x_n}{r_n}\right|^{1-\alpha} \le C \left|\frac{y_n-x_n}{r_n}\right|^{1-\alpha} \to 0,
\end{align*}
a contradiction which proves that $z_n$ cannot tend to $0$. We infer that the limit function $v_1$ is non-constant, and in particular $|v_1(0) -v_1(z_\infty)|\neq 0$ for $z_\infty=\lim z_n$. It is easy to see that this leads again to a contradiction, as by the assumption in \eqref{conditions r_n} and the uniform convergence of $\mf{v}_n$ on compact sets of $\R^N$ (recall that $\Omega_n'$ tends to a hal-space, but $\Omega_n'\subset \Omega_n \to \R^N$, and the function $\mf{v}_n$ is defined in $\Omega_n$) we have
\begin{align*}
|v_1(z_\infty)| &= \lim_n|v_{1,n}(z_n)| = \lim_n \frac{\eta(y_n) |u_{1,n}(y_n)|}{L_n r_n^{\alpha}} \le \lim_n \frac{m l r_n^{1-\alpha}}{L_n } \frac{\dist(y_n,\pa B_2)}{r_n} \\
&\le  \lim_n \frac{m l r_n^{1-\alpha}}{L_n } \left(\frac{\dist(x_n,\pa B_2)}{r_n}+ \frac{|x_n-y_n|}{r_n}\right) = 0,
\end{align*}
where we recall that $(\dist(x_n,\pa B_2)/r_n)$ is bounded, and $m$ denotes the upper bound on the $L^\infty$ norm of $\{\mf{u}_n\}$ in $B_3$. With similar (actually easier) computations one can also check that $|v_1(0)|=0$, reaching in this way the sought contradiction.
\end{proof}

\begin{lemma}\label{lem 3.6 notateve}
Let $r_n:=|x_n-y_n|$. Then there exist $\mf{v} \in \mathcal{C}^{0,\alpha}(\R^N)$ such that up to a subsequence
\begin{itemize}
\item[($i$)] $\mf{v}_n \to \mf{v}$ uniformly on compact sets of $\R^N$;
\item[($ii$)] $\mf{v}_n \to \mf{v}$ in $H^1_{\loc}(\R^N)$, and for every $r>0$
\[
\lim_{n \to \infty} \int_{B_r} M_n |v_{i,n}|^{p+1} |v_{j,n}|^{p+1}  = 0 \qquad \text{for every $(i,j) \in \mathcal{K}_2$}.
\]
\end{itemize}
\end{lemma}
\begin{proof}
First of all, we show that in the present setting $\Omega_n' \to \R^N$. Indeed, by definition and using the Lipschitz continuity of $\eta$ we have
\[
L_n = \frac{|(\eta u_{1,n})(x_n)-(\eta u_{1,n})(y_n)|}{r_n^{\alpha}} \le \frac{ml}{r_n^{\alpha}} \left(\dist(x_n,\pa B_2) + \dist(y_n,\pa B_2) \right).
\]
Equivalently,
\[
 \left(\frac{\dist(x_n,\pa B_2)}{r_n} + \frac{\dist(y_n,\pa B_2)}{r_n} \right) \ge \frac{L_n r_n^{\alpha-1}}{ml} \to +\infty
 \]
 as $n \to \infty$, which proves the assertion. The rest of the proof is now an easy generalization of that of Lemma 3.6 in \cite{nttv}, and thus is only sketched. 
 
With our choice of $r_n$, by Lemma \ref{M_n illimitato} we have $M_n \to +\infty$, and the assumption of Lemma \ref{lem: bound in 0} are satisfied. Therefore, $\{\mf{v}_n(0)\}$ is a bounded sequence, which by point (5) of Lemma \ref{lem: basic prop} implies that $\mf{v}_n \to \mf{v}$ locally uniformly on $\R^N$ (up to a subsequence). 

For point ($ii$), we introduce a smooth cut-off function $\varphi$ with $0 \le \varphi \le 1$, $\varphi \equiv 1$ in $B_r$ and $\varphi \equiv 0$ in $\R^N \setminus B_{2r}$. Testing the equation for $v_{i,n}$ against $\varphi$ and using the the Kato inequality, it is not difficult to check that 
\begin{equation}\label{bound interaction 431}
\int_{B_r} M_n a_{ij} |v_{i,n}|^{p} \sum_{j \neq i} a_{ij} |v_{j,n}|^{p+1} \le C \qquad \forall i,
\end{equation}
and since $M_n \to +\infty$ this implies $v_i v_j \equiv 0$ in $\R^N$ whenever $(i,j) \in \mathcal{K}_2$ (recall that $\mathcal{K}_2$ has been defined in \eqref{def indexes}). As a consequence
\begin{multline*}
M_n \int_{B_r} |v_{i,n}|^{p+1}  |v_{j,n}|^{p+1}  \le  \|v_{i,n}\|_{L^\infty(B_r \cap \{v_{i} \equiv 0\})} \int_{B_r} M_n |v_{i,n}|^{p}  |v_{j,n}|^{p+1} \\
+   \|v_{j,n}\|_{L^\infty(B_r \cap \{v_{j} \equiv 0\})} \int_{B_r} M_n |v_{i,n}|^{p+1}  |v_{j,n}|^{p} \to 0
\end{multline*}
as $n \to \infty$, for every $(i,j) \in \mathcal{K}_2$. 

It remains to prove that $\mf{v}_n \to \mf{v}$ strongly in $H^1_{\loc}(\R^N)$. To this aim, we test the equation for $v_{i,n}$ against $v_{i,n} \varphi^2$, deducing that $\|\nabla v_{i,n}\|_{L^2(B_r)}$ is a bounded sequence. This ensures that $v_{i,n} \rightharpoonup v_i$ in $H^1(B_r)$, and that, if necessary replacing $r$ with a slightly smaller quantity, also $\| \nabla v_{i,n}\|_{L^2(\pa B_r)}$ is bounded. Hence, testing the equation for $v_{i,n}$ against $v_{i,n}-v_i$, and recalling also \eqref{bound interaction 431}, we conclude that 
\begin{align*}
\lim_{n \to \infty} \left| \int_{B_r} |\nabla v_{i,n}|^2 - |\nabla v_i|^2 \right| &= \lim_{n \to \infty} \left|\int_{B_r} \nabla v_{i,n} \cdot \nabla (v_{i,n}-v_i) \right|\\
& \le \lim_{n \to \infty} \|v_{i,n}-v_i\|_{L^\infty(B_r)} \left( \int_{\pa B_r} |\pa_\nu v_{i,n}| + C \right) = 0
\end{align*}
as $n \to +\infty$, i.e. $v_{i,n} \to v_i$ also in the $H^1(B_r)$ norm, which completes the proof.
\end{proof}

\begin{lemma}\label{lem: end blow-up}
Let $\mf{v}$ be defined in Lemma \ref{lem 3.6 notateve}. Then:
\begin{itemize}
\item[($i$)] $v_i v_j \equiv 0$ for every $(i,j) \in \mathcal{K}_2$;
\item[($ii$)] $\max_{x \in \partial B_1} |v_1(x)-v_1(0)|=1$;
\item[($iii$)] it results
\[
-\Delta v_i=0 \qquad \text{ in } \left\{\sum_{j\in I_h} |v_j|>0 \right\}
\]
for every $i\in I_h$, $h=1,\ldots, m$.
\item[($iv$)] $v_j \equiv 0$ in $\R^N$ for every $j \not \in I_1$;
\item[($v$)] the set $\{x\in \Omega:\,v_i(x)=0 \text{ for all $i \in I_1$}\}$ is not empty, and the sets $\{x\in \Omega:\,v_i(x)\neq 0\}$ are connected for every $i \in I_1$. In particular, $v_i$ does not change sign for every $i \in I_1$.
\end{itemize}
\end{lemma}
\begin{proof}
The first two points are trivial. Concerning ($iii$), by continuity the set $\left\{\sum_{j\in I_h} |v_j|>0 \right\}$ is open. Given any point $x_0$ such that $\sum_{j\in I_h} |v_j(x_0)|>0$, we find a neighbourhood of $x_0$ where $v_i$ is harmonic for $i\in I_h$. By H\"older continuity there exists $\rho>0$ small enough that $\sum_{j\in I_h} |v_j| \ge 2\gamma >0$ in $B_{\rho}(x_0)$, so that by uniform convergence $\sum_{j\in I_h} |v_{j,n}(x_0)| \ge \gamma$ in $B_{\rho}(x_0)$ for every $n$ sufficiently large. Therefore, for any $i\in I_h$ and $k\notin I_h$,
\[
\int_{B_{\rho}(x_0)} M_n |v_{i,n}|^p |v_{k,n}|^{p+1} \le   C\sum_{j\in I_h}\int_{B_{\rho}(x_0)} M_n |v_{j,n}|^{p+1} |v_{k,n}|^{p+1} \to 0
\]
as $n \to \infty$, for every $j$ such that $(i,j) \not \in \mathcal{K}_1$. Testing the equation for $v_{i,n}$ against a test function $\varphi \in \mathcal{C}^\infty_c(B_{\rho}(x_0))$, we obtain (recall that $a_{ij}=0$ whenever $(i,j)\in \mathcal{K}_1$)
\[
\int_{B_{\rho}(x_0)} \nabla v_{i,n} \cdot \nabla \varphi = \int_{B_{\rho}(x_0)} g_{i,n} \varphi - M_n |v_{i,n}|^{p-1}v_{i,n} \sum_{j \neq i} a_{ij} |v_{j,n}|^{p+1} \varphi
\]
and, as $n \to \infty$,
\[
\int_{B_{\rho}(x_0)} \nabla v_i \cdot \nabla \varphi = 0,
\]
which completes the proof.\\
As far as ($iv$) is concerned, by the previous point $v_1$ must vanish somewhere in $\R^N$ (indeed, if not, $v_1$ would be a non-constant H\"older continuous harmonic function in $\R^N$, a contradiction by Corollary \ref{lem:liouville1}), and also $v_j$ must vanish somewhere for every $j \not \in I_1$ (otherwise we would have $v_1\equiv0$ in $\R^N$, again a contradiction). This, by continuity, implies that $|v_1|$ and $|v_j|$ must have a common zero, and thus they satisfy all the assumptions of Lemma \ref{lem:liouville2}. Since $v_1$ is not constant, we deduce that
\[
v_j \equiv 0\qquad\text{in $\R^N$ for every $j \not \in I_1$}.
\]
To prove point ($v$) we argue by contradiction assuming that  $\{v_1 \neq 0\}$ non-trivially decomposes into
$\O_1\cup\O_2$. Then one of the pairs $(v_1|_{\O_1}^+,v_1|_{\O_2}^+)$, $(v_1|_{\O_1}^-,v_1|_{\O_2}^-)$, $(v_1|_{\O_1}^+,v_1|_{\O_2}^-)$ and $(v_1|_{\O_1}^-,v_1|_{\O_2}^+)$ - extended by $0$ to the whole $\R^N$ - would be non-trivial and would satisfy the assumptions of Lemma \ref{lem:liouville2}, a contradiction.
\end{proof}

\subsection{Almgren monotonicity formula}

As in \cite{nttv}, to complete the proof Theorem \ref{thm: holder bounds} we show that $v_1$ is radially homogeneous with respect to each one of its zeros. To this aim, we state an Almgren monotonicity formula for the elements $\mf{v}_n$ of the blow-up sequence, and we show that the limit function $\mf{v}$ inherits such property. 

We recall that $\mf{v}_n$ is a solution to \eqref{system blow-up}. Let $x_0 \in \Omega$ and $r>0$ such that $B_r(x_0) \Subset \Omega_n$; we define
\[
\begin{split}
  \bullet \quad & H_n(x_0,r):= \frac{1}{r^{N-1}} \int_{\partial B_r(x_0)} \sum_{i=1}^k v_{i,n}^2\\
  \bullet \quad & E_n(x_0,r):= \frac{1}{r^{N-2}} \int_{B_r(x_0)} \sum_{i=1}^k |\nabla v_{i,n}|^2+ 2 M_n\sum_{1\le i<j\le k} a_{ij} |v_{i,n}|^{p+1} |v_{j,n}|^{p+1} - \sum_{i=1}^k g_{i,n}(x) v_{i,n}\\
  \bullet \quad & N_n(x_0,r):= \frac{E_n(x_0,r)}{H_n(x_0,r)} \qquad (\text{Almgren frequency function}).
\end{split}
\]
We also set
\[
\begin{split}
  \bullet \quad & H_{\infty}(x_0,r):= \frac{1}{r^{N-1}} \int_{\partial B_r(x_0)} \sum_{i=1}^k v_{i}^2\\
  \bullet \quad & E_{\infty}(x_0,r):= \frac{1}{r^{N-2}} \int_{B_r(x_0)} \sum_{i=1}^k |\nabla v_{i}|^2\\
  \bullet \quad & N_{\infty}(x_0,r):= \frac{E_{\infty}(x_0,r)}{H_{\infty}(x_0,r)} \qquad (\text{Almgren frequency function}).
\end{split}
\]

Parts of the proofs of the following results can be obtained by slightly modifying those of Proposition 3.9 in \cite{nttv} (where a specific choice of the reaction terms is considered), of the results of Section 2 in \cite{tt} (where segregated configurations are considered), or of the results in Subsection 3.1 in \cite{SoZi} (where the reaction term $f_{i,\beta}(x)$ is replaced by $f_{i,\beta}(x,u_i)$). We will only prove what requires something new.

Since the limit function $\mf{v}$ is non-trivial and continuous, there exists $0<r_1<r_2$ and $x_0 \in \R^N$ such that $H(x_0,r) \neq 0$ for every $r \in (r_1,r_2)$.

\begin{lemma}\label{lem: derivatives}
Let $r \in (r_1,r_2)$. Then
\[
\frac{d}{d r}H_n(x_0,r) = \frac{2}{r^{N-1}} \int_{\pa B_r(x_0)} \sum_{i=1}^d v_{i,n} \pa_\nu v_{i,n} = \frac{2 E_n(x_0,r)}{r},
\]
and 
\begin{multline*}
N_n(x_0,r+\delta) - N_n(x_x,r)  = \int_{r}^{r+\delta}  \frac{2}{s^{2N-3} H_n(x_0,s)} \left[\left(\int_{\pa B_s(x_0)} \sum_i (\pa_\nu v_{i,n})^2 \right) \left( \int_{\pa B_s(x_0)} \sum_i v_{i,n}^2 \right) \right. \\
\left. - \left( \int_{\pa B_s(x_0)} \sum_i v_{i,n} \pa_{\nu} v_{i,n} \right)^2\right]  +o_n(1),
\end{multline*}
where $o_n(1) \to 0$ as $n \to \infty$, whenever $\delta$ is such that $r+\delta \in (r_1,r_2)$.
\end{lemma}
\begin{proof}
Being $a_{ij}=0$ for every $(i,j) \in \mathcal{K}_1$ (see definition \eqref{def indexes}), we can directly repeat the proof of Lemma 3.3 in \cite{SoZi}, obtaining
\begin{align*}
\frac{\de}{\de r} & N_n(x_0,r)  =\\
&  \frac{2}{r^{2N-3} H_n(x_0,r)} \left[\left(\int_{\pa B_r(x_0)} \sum_i (\pa_\nu v_{i,n})^2 \right) \left( \int_{\pa B_r(x_0)} \sum_i v_{i,n}^2 \right) - \left( \int_{\pa B_r(x_0)} \sum_i v_{i,n} \pa_{\nu} v_{i,n} \right)^2\right] \\
& +  \frac{\left(4-\frac{2pN}{p+1}\right)M_n}{H_n(x_0,r) r^{N-1}} \int_{B_r(x_0)} \sum_{i<j} a_{ij} |v_{i,n}|^{p+1} |v_{j,n}|^{p+1}  \\
& + \frac{2pM_n}{(p+1)H_n(x_0,r) r^{N-2}}  \int_{\pa B_r(x_0)} \sum_{i<j} a_{ij} |v_{i,n}|^{p+1} |v_{j,n}|^{p+1} \\
& + \frac{1}{H_n(x_0,r) r^{N-1}} \int_{B_r(x_0)} \left[\sum_i g_{i,n}(x) v_{i,n} + 2\sum_i g_{i,n}(x) \nabla v_{i,n} \cdot (x-x_0) \right] \\
& -\frac{1}{r^{N-2}H_n(x_0,r)} \int_{\pa B_r(x_0)} g_{i,n}(x) v_{i,n}.
\end{align*}
The thesis follows thanks to point (3) of Lemma \ref{lem: basic prop} and to point ($ii$) of Lemma \ref{lem 3.6 notateve}, having observed that for every $\delta>0$ the function $H_n(x_0,\cdot)$ is uniformly bounded from below in $[r_1+\delta,r_2-\delta]$. 
\end{proof}

The main consequences of the previous lemma are summarized in the following statement.

\begin{proposition}\label{prop: consequence almgren}
For every $x_0 \in \R^N$ we have that $H_{\infty}(x_0,r) \neq 0$ for every $r>0$; the function $N_\infty(x_0,\cdot)$ is absolutely continuous and monotone non-decreasing, and
\[
\frac{\de}{\de r} \log H_{\infty}(x_0,r) = \frac{2}{r}N_{\infty}(x_0,r).
\]
Moreover, if $N_{\infty}(x_0,r)= \gamma$ for every $r \in [\rho_1,\rho_2]$, then $\mf{v}=r^\gamma \hat{\mf{v}}(\theta)$ in $\{\rho_1 < r< \rho_2\}$, where $(r,\theta)$ denotes a system of polar coordinates centred in $x_0$.
\end{proposition}

\begin{proof}
The result can be proved as in steps 4, 5 and 6 of Proposition 3.9 in \cite{nttv}, and thus here we only sketch the argument. 

Given $x_0\in \R^N$, let $r_1<r_2$ be such that $H_\infty(x_0,r)\neq 0$ in $(r_1,r_2)$. By Lemma \ref{lem: derivatives}, we have 
\begin{multline*}
N_n(x_0,r+\delta) - N_n(x_0,r) =   \int_{r}^{r+\delta} \frac{2}{s^{2N-3} H_n(x_0,s)} \left[ \left(\int_{\pa B_s(x_0)} \sum_i (\pa_\nu v_{i,n})^2 \right) \left( \int_{\pa B_s(x_0)} \sum_i v_{i,n}^2 \right) \right. \\
\left. - \left( \int_{\pa B_s(x_0)} \sum_i v_{i,n} \pa_{\nu} v_{i,n} \right)^2\right] + o_n(1),
\end{multline*}
where $o_n(1) \to 0$ as $n \to \infty$, for any $r,\delta$ such that $r,r+\delta \in (r_1,r_2)$. Passing to the limit in the previous identity, we obtain
\begin{multline}\label{der di N integrata}
N_\infty(x_0,r+\delta) - N_\infty(x_0,r) =   \int_{r}^{r+\delta} \frac{2}{s^{2N-3} H_\infty(x_0,s)} \left[ \left(\int_{\pa B_s(x_0)} \sum_i (\pa_\nu v_{i})^2 \right) \left( \int_{\pa B_s(x_0)} \sum_i v_{i}^2 \right) \right. \\
\left. - \left( \int_{\pa B_s(x_0)} \sum_i v_{i} \pa_{\nu} v_{i} \right)^2\right],
\end{multline}
and the right hand side is nonnegative by the Cauchy-Schwarz inequality. This proves the monotonicity of $N_\infty(x_0,\cdot)$. 

To show that $H_{\infty}(x_0,r) \neq 0$ for every $r>0$, we first observe that by Lemma \ref{lem: derivatives} the function $H_\infty(x_0,\cdot)$ is non-decreasing in $r$ when $H_\infty(x_0,r) \neq 0$. Thus, if $H_\infty(x_0,r) = 0$ for some positive $r$, it is well defined the number $0<r_0:= \inf\{r>0: H_\infty(x_0,r) \neq 0\}$, and $H_\infty(x_0,r)>0$ for every $r>r_0$. On the other hand, by the monotonicity of $N_\infty(x_0,\cdot)$, we have also
\[
\frac{d}{dr}\log H_\infty(x_0,r) = \frac{2N_\infty(x_0,r)}{r} \le \frac{C}{r} \quad \Longrightarrow \quad H_\infty(x_0,r_2) \le H_\infty(x_0,r_1) \left(\frac{r_2}{r_1}\right)^{2C}
\]
for every $r_1,r_2 \in (r_0,r_0+1)$; taking the limit as $r_1 \to r_0^+$, by continuity, we infer that $H_\infty(x_0,r_2) = 0$ for every $r_2 \in (r_0,r_0+1)$, a contradiction.

It remains to prove that if $N_\infty(x_0,r) \equiv \gamma$ is constant on an interval $r \in (\rho_1,\rho_2)$, then the function $\mf{v}$ is radially homogeneous. To this aim, we observe that in such case the right hand side in \eqref{der di N integrata} is necessarily $0$ for almost every $r$, which, by the Cauchy-Schwarz inequality, is possible only if 
\[
	(x-x_0) \cdot \nabla \mf{v}_\infty (x) = \lambda(x - x_0) \mf{v}_\infty 
\]
Inserting this relation in the definition of $N_\infty(x_0,r)$, we can directly compute $\lambda(x-x_0) = \gamma$ and the thesis follows.
\end{proof}

%
%
%
%
%
%
%

\subsection{Conclusion of the proof of the uniform H\"older bounds}

Using Lemma \ref{lem: end blow-up} and Proposition \ref{prop: consequence almgren} we can complete the proof of Theorem \ref{thm: holder bounds}.

We recall that $(v_1,\dots,v_d)$ is globally $\alpha$-H\"older continuous, and, by Proposition \ref{lem: end blow-up}, it is possible to choose $x_0$ such that $v_i(x_0) = 0$ for every $i$. We claim that $N_\infty(x_0,r) \equiv \alpha$ for $r>0$. Indeed, if $N_\infty(x_0,\bar r) \le \alpha -\eps$ for some $\eps>0$, then by monotonicity
\[
\frac{\de}{\de r} H_\infty(x_0,r) = \frac{2}{r}N_\infty(x_0,r) \le \frac{2 (\alpha-\eps)}{r}
\]
for every $r \in (0,\bar r)$, which implies $H(r) \ge C r^{2(\alpha-\eps)}$ for $0<r<\bar r$. On the contrary, by H\"older continuity and the fact that $v_i(x_0) = 0$ for all $i$ we have also $H_\infty(x_0,r) \le C r^{2\alpha}$ for all $r>0$, a contradiction for $r$ small. Arguing in a similar way for $r$ large it is possible to rule out the possibility that $H_\infty(x_0,\bar r) \ge \alpha +\eps$ for some $\bar r,\eps>0$. 

As a consequence $N_\infty(x_0,r) \equiv \alpha$, whence thanks to Proposition \ref{prop: consequence almgren} we deduce that $v_1(x) = r^\alpha g_1(\theta)$. Therefore, the zero set $\Gamma = \{v_1 = 0\}$ is a cone with respect to any of its points, i.e. is an affine subspace of $\R^N$. Now there are two cases: either the dimension of $\Gamma$ is equal to $N-1$, or it is smaller than $N-1$. In the former case, $v_1$ is a positive harmonic $\alpha$-H\"older continuous function in a half-space. We extend it by odd symmetry in the all of $\R^N$, obtaining a sign-changing globally $\alpha$-H\"older continuous harmonic function in $\R^N$, in contradiction with Corollary \ref{lem:liouville1}. If on the contrary the dimension of $\Gamma$ is smaller than $N-1$, then $v_1$ is harmonic in $\R^N$ minus a set of zero capacity, so that $v_1$ is a nonconstant nonnegative $\alpha$-H\"older continuous harmonic function in $\R^N$, again a contradiction.

\subsection{Uniform H\"older bounds at the boundary}\label{fino al bordo} We now consider the case of uniform H\"older bounds at the boundary of $\Omega$, for a smooth domain, that is, we give a proof of Theorem \ref{thm: holder bounds boundary}. We still consider solutions $\mf{u}_\beta$ of the system \eqref{eq:main_system}, under the same assumptions of the interior H\"older bounds; moreover, on (a portion of) the boundary of $\Omega$, we assume that $\mf{u}_\beta = 0$. In particular, we assume that $\mf{u}_\beta$ solve
\[
	\begin{cases}
		-\Delta u_i=f_{i,\beta}-\beta \mathop{\sum_{j=1}^d}_{j\neq i} a_{ij} u_i |u_i|^{p-1}|u_j|^{p+1} & \text{ in $\Omega$,} \\
		u_{i,\beta} = 0 &\text{ on $\partial \Omega\cap B_3$}
	\end{cases}
	\qquad i=1,\ldots, d.
\]
For $\eta \in \Ccal_c^1(\R^N)$ as in \eqref{def eta}, we wish to show that uniform bounds in $L^{\infty}(B_3)$ of $\{\mf{u}_\beta\}$ imply that the function $\{\eta \mf{u}_\beta\}$ are uniformly bounded in $\Ccal^{0,\alpha}(B_3)$ for any $\alpha \in (0,1)$.

The proof is based on a contradiction argument, much similar to the proof that we gave for the interior estimates. Indeed, until Lemma \ref{lem 3.6 notateve}, the two proofs coincide. At that point we have to distinguish the possible behaviours of the scaled sets $\Omega_n := (B_2 \cap \Omega - x_n) / r_n$: choosing $r_n = |x_n - y_n|$, in the case of interior estimate, we already knew that
\[
	\frac{\dist(x_n, \partial (\Omega \cap B_2))}{r_n} \to \infty, 
\]
that is, the scaled domains exhausted $\R^N$; this conclusion followed by our specific choice of $\eta$. In the present setting, it may happen that the scaled domains converge to an half plane, as consequence of the presence of the boundary of $\Omega$, where the functions $\mf{u}_\beta$ assume their null Dirichlet boundary condition. To roll out this scenario, we consider the following result.
\begin{lemma}
We have 
\[
	\lim_{n \to \infty} \frac{ \min(\dist(x_n, \partial \Omega), \dist(y_n, \partial \Omega))}{|x_n - y_n|} = +\infty.
\]
\end{lemma}
\begin{proof}
By contradiction, if for example
\[
	\frac{ \dist(x_n, \partial \Omega)}{|x_n - y_n|} \leq C
\]
then there exists a sequence $x_{0,n} \in \R^N$, $|x_{0,n}| \leq C$ such that
\[
	x_n + x_{0,n} |x_n - y_n| \in \partial \Omega \qquad \text{and } \mf{v}_n(x_{0,n}) = 0.
\]
Let $r_n = |x_n - y_n|$. Up to a subsequence, using Lemma \ref{lem: basic prop}-(1) and -(4), we see that there exists $\mf{v} \in \Ccal^{0,\alpha}(\R^N)$ such that
\[
	\mf{\bar v}_n \to \mf{v} \quad \text{in $\Ccal^{0,\alpha}_\loc$, } \mf{v}_n \to \mf{v} \quad \text{locally uniformly in $\R^N$}.
\]
Moreover, up to a translation and a rotation, we may assume that $\mf{v} = 0$ in the half space $\{x \cdot e_1 \leq 0\}$. Moreover, thanks to our choice for $r_n$, at least one component of $\mf{v}$ is nontrivial. Without loss of generality, let us assume that $v_{1} \neq 0$. Regardless of the behaviour of $M_n$, by the Kato inequality we see that
\[
	-\Delta |v_1| \leq 0, \quad |v_1| \geq 0 \quad \text{and $|v_1| = 0$ in $\{x \cdot e_1 \leq 0\}$}.
\]
Letting $w_1(x) = |v_1(x - 2 (x \cdot e_1) e_1)|$ and applying Lemma \ref{lem:liouville2}, we find the desired contradiction.
\end{proof}
Let us observe that in the previous proof, we did not use the variational structure of the system. Now that we have established that the boundary of $\Omega$ is far from the points $x_n$ and $y_n$, the proof runs as in the standard case.

\section{Properties of the limit profiles}\label{sec: properties}


We shall now improve the regularity results so far obtained for the functions in the family $\{\mf{u}_\beta\}_\beta$ and, in particular, we aim at showing that, under a little more restrictive assumption on the nonlinearities $f_{i,\beta}$, any limit of the family (as $\beta \to \infty$) is an element of the class $\Gcal(\Omega)$. In order to verify the previous claim (and, as a consequence, Theorem \ref{thm: consequences}), we shall prove several intermediate results.

First, using the information that the functions $\{\mf{u}_\beta\}_\beta$ constitute a family which is uniformly bounded in the $\Ccal^{0,\alpha}$-norm, as a direct consequence of the Ascoli-Arzela compactness criterion, we can show that
\begin{lemma}\label{lem: cu and int 0}
Under the same assumptions of Theorem \ref{thm: holder bounds}, up to a subsequence we have that there exists a limiting configuration $\mf{u} \in H^1 \cap \Ccal(\Omega)$ such that
\[
	\mf{u}_\beta \to \mf{u} \qquad \text{strongly in $H^1 \cap \Ccal^{0,\alpha}(K)$, for all $\alpha\in (0,1)$}
\]
for any set $K \Subset \Omega$. Moreover
\begin{enumerate}
	\item the components of $\mf{u}$ are segregated in groups, that is, $u_i u_j \equiv 0$ in $\Omega$ for every $(i,j) \in \mathcal{K}_2$;
	\item for any $K \Subset \Omega$, we have 
	\[
		\beta \int_{K} |u_{i,\beta}|^{p+1} |u_{j,\beta}|^{p+1} \to 0 \qquad \text{ for every $(i,j) \in \mathcal{K}_2$;}
	\]
	\item for $ i\in I_h$, each component $u_i$ satisfies
	\[
		-\Delta u_i = f_i(x, \mf{u})  \qquad \text{ in } \left\{\sum_{j\in I_h} |u_j|>0\right\}.
	\]
\end{enumerate}
\end{lemma}
\begin{proof}
Most of the details of the proof have already been encountered in the previous section: we point out also \cite[Theorem 1.4]{nttv} and Lemma \ref{lem 3.6 notateve} for similar computations.
\end{proof}

Next, under an additional assumption of the nonlinearity $f_{i,\beta}$, we shall show that the variational structure of the original system, in a sense, passes to the limit together with the functions $\{\mf{u}_\beta\}_\beta$. This strong property of the limiting function $\mf{u}$ is rigorously stated as the validity of the Pohozaev identity.
\begin{lemma}
Let $\mf{u}$ be in the limit class of $\{\mf{u}_\beta\}_\beta$. Let us assume that there exist $f_i \in \Ccal(\Omega \times \R^d)$, $i=1,\dots, k $, such that $f_{i,\beta} \to f_i$ in $\Ccal_\loc(\Omega \times \R^d)$. Then for every $x_0 \in \Omega$ and a.e.~$0<r<\dist(x_0,\partial \Omega)$ it holds
\begin{align*}
(2-N) \sum_{i=1}^d \int_{B_r(x_0)} |\nabla u_i|^2 &= r \sum_{i=1}^d \int_{\pa B_r(x_0)} \left( 2 (\pa_{\nu} u_i)^2 -|\nabla u_i|^2 \right) \\
& \qquad + 2 \sum_{i=1}^d \int_{B_r(x_0)} f_{i}(x,\mf{u}) \nabla u_i \cdot (x-x_0).
\end{align*}
\end{lemma}
\begin{proof}
In order to prove the result, it is sufficient to prove the validity of similar identities of the original functions $\{\mf{u}_\beta\}_\beta$, and then exploit the strong convergence properties of the family to conclude. In particular, under the assumption of the lemma, multiplying the equation \eqref{eq:main_system} with $\nabla u_{i,\beta} \cdot (x-x_0)$ and integrating by parts over $B_r(x_0)$ (we recall once again that the function $\mf{u}_\beta$ is, by standard regularity argument, a $\mathcal{C}^{1,\alpha}$-solution of \eqref{eq:main_system} for every $0<\alpha<1$), we obtain
\begin{multline*}
	(2-N) \sum_{i=1}^d \int_{B_r(x_0)} |\nabla u_{i,\beta}|^2 = r \sum_{i=1}^d \int_{\pa B_r(x_0)} \left( 2 (\pa_{\nu} u_{i,\beta})^2 -|\nabla u_{i,\beta}|^2 \right)\\
	+ 2 \sum_{i=1}^d \int_{B_r(x_0)} f_{i,\beta}(x,\mf{u}) \nabla u_{i,\beta} \cdot (x-x_0) 	+ \int_{B_r(x_0)} \beta N \sum_{i \neq j} a_{ij} |u_{i,\beta}|^{p+1} |u_{j,\beta}|^{p+1} \\
	- r \int_{\partial B_r(x_0)} \beta \sum_{i \neq j} a_{ij} |u_{i,\beta}|^{p+1} |u_{j,\beta}|^{p+1}.
\end{multline*}
The conclusion now follows from Lemma \ref{lem: cu and int 0}-(2).
\end{proof}

A deep consequence of the variational structure of the limiting system is expressed by the Almgren's monotonicity formula. From now on, we assume that the limiting profile $\mf{u}$ is non trivial, since otherwise all the following results are tautologically true.

Similarly to the previous section, we define, for $x_0\in\Omega$ and $r>0$ small,
\[
E(x_0,\mf{u},r)=\frac{1}{r^{N-2}}\sum_{i=1}^d\int_{B_r(x_0)} (|\nabla u_i|^2-f_i(x,\mf{u})u_i),\qquad H(x_0,\mf{u},r)=\frac{1}{r^{N-1}}\sum_{i=1}^d \int_{\partial B_r(x_0)} u_i^2
\]
and, whenever it makes sense, the Almgren's quotient by
\[
N(x_0,\mf{u},r)=\frac{E(x_0,\mf{u},r)}{H(x_0,\mf{u},r)}.
\]
We have
\begin{theorem}\label{thm:Almgren}
There exists $C>0$ for which the following holds: for every $\tilde \Omega\Subset \Omega$ there exists $\tilde r>0$ such that for every $x_0\in \tilde \Omega$ and $r\in (0,\tilde r]$ we have $H(x_0,(u,v),r)\neq 0$, $N(x_0,(u,v),\cdot)$ is absolutely continuous function, and
$$
\frac{d}{dr}N(x_0,\mf u,r)\geq -2Cr(N(x_0,\mf u,r)+1).
$$
In particular, $e^{C r^2}(N(x_0,\mf u,r)+1)$ is a non decreasing function for $r\in (0,\tilde r]$ and the limit $N(x_0,\mf u,0^+):=\lim_{r\to 0^+} N(x_0,\mf u,r)$ exists and is finite. Also,
$$
\frac{d}{dr}\log(H(x_0,\mf u,r))=\frac{2}{r}N(x_0,\mf u,r)\qquad \forall r\in (0,\tilde r).
$$
\end{theorem}
\begin{proof} One can follow exactly the proof of Theorem 3.21 in \cite{rtt}, observing that $C>0$ is a constant such that
\[
\frac{1}{r^{N}}\sum_{i=1}^d \int_{B_r(x_0)} f_i(x,\mf{u}) \leq C (E(x_0,\mf{u},r)+H(x_0,\mf{u},r))
\]
for every $r>0$ small enough, $x_0\in \Omega$ (compare with Lemma 3.19 in \cite{rtt}).
Such inequality holds since, for each $i\in I_h$, by assumption (G2),
\begin{align*}
\frac{1}{r^{N}} \int_{B_r(x_0)}f_i(x,\mf{u}) &\leq \frac{C_1}{r^{N}}\sum_{j\in I_h} \int_{B_r(x_0)}  u_j^2\\
	&\leq C_1' \sum_{j\in I_h} \left( \frac{1}{r^{N-2}} \int_{B_r(x_0)} |\nabla u_j|^2\, dx+\frac{1}{r^{N-1}}\int_{\partial B_r(x_0)}u_j^2 \right)
\end{align*}
by the Poincar\'e inequality, and hence, summing up for every $i\in I_h$ and for $h=1,\ldots, m$,
\[
\frac{1}{r^{N}} \sum_{i=1}^d\int_{B_r(x_0)}f_i(x,\mf{u}) \leq \frac{C_2}{N-1} \sum_{i=1}^d \left( \frac{1}{r^{N-2}} \int_{B_r(x_0)} |\nabla u_i|^2\, dx+\frac{1}{r^{N-1}}\int_{\partial B_r(x_0)}u_i^2 \right).
\] 
Next we  observe that
\begin{equation*}
\begin{split}
\frac{1}{r^{N-2}} \sum_{i=1}^k \int_{B_r(x_0)}  |\nabla u_i|^2 =E(x_0,\mf u,r)+\frac{1}{r^{N-2}}\sum_{i=1}^k \int_{B_r(x_0)}f_i(x,\mf{u}) u_i.
\end{split}
\end{equation*}
Thus for $r$ small enough such that $\frac{C_2}{N-1}  r^2\leq 1/2$ the result follows, with $C=\frac{2C_2}{N-1}$. 
\end{proof}

%

We are now in a position to conclude with the last result of the present section: in the following proposition, we show that any segregated $H^1(\Omega)$ solution which belongs to $\Ccal^{0,\alpha}(\Omega)$ for any $\alpha \in (0,1)$ and also satisfies the Pohozaev identity, is actually more regular and belongs to $Lip(\Omega)$.

\begin{proposition}\label{prp: lip}
Let $\mf{u}=(u_1,\dots,u_d) \in H^1(\Omega, \R^d) \setminus \{\mf{0}\}$ be such that:
\begin{itemize}
\item $\mf{u} \in \Ccal^{0,\alpha}(\Omega)$ for any $\alpha \in (0,1)$, and such that $u_i u_j \equiv 0$ in $\Omega$ for every $(i,j) \in \mathcal{K}_2$;
\item for $ i\in I_h$, each component $u_i$ satisfies the compatibility condition
	\[
		-\Delta u_i = f_i(x, \mf{u})  \qquad \text{ in } \left\{\sum_{j\in I_h} |u_j|>0\right\},
	\]
where there exists $C>0$ such that
\[
	\sup_{i \in I_h} \sup_x  \left|\frac{f_i(x,\mf{s})}{\sum_{j \in I_h} |s_j|} \right|  \le C
\]
for every $\mf{s} \in [0,1]^N$, for every $h$;
\item for every $x_0 \in \Omega$ and $0<r<\dist(x_0,\partial \Omega)$ it holds
\begin{align*}
(2-N) \sum_{i=1}^d \int_{B_r(x_0)} |\nabla u_i|^2 &= r \sum_{i=1}^d \int_{\pa B_r(x_0)} \left( 2 (\pa_{\nu} u_i)^2 -|\nabla u_i|^2 \right) \\
& \qquad + 2 \sum_{i=1}^d \int_{B_r(x_0)} f_{i}(x,\mf{u}) \nabla u_i \cdot (x-x_0).
\end{align*}
\end{itemize}
Then $\mf{u} \in Lip(\Omega)$.
\end{proposition}

The proof is based on the simple observation that a function $\mf{u} \in H^1(\Omega)$ is locally Lipschtiz continuous if and only if for any $K \Subset \Omega$ there exists a constant $C>0$ and a radius $0 < \bar r < \dist(K,\partial \Omega)$ such that for any $x_0 \in K$ and $0<r< \bar r$ it holds
\begin{equation}\label{eqn: morrey}
	\frac{1}{r^N} \sum_{i=1}^{d} \int_{B_r(x_0)} |\nabla u_i|^2 \leq C.
\end{equation}
In order to show that the previous inequality is true, we proceed with several steps in the same way as \cite[Section 4]{nttv}: we refer to that paper for the omitted details in the proofs. First, we recall that
\[
	\Gamma_\mf{u} := \{x \in \Omega : \mf{u}(x) = 0 \}.
\]
Let $K \Subset \Omega$ be a fixed subset of $\Omega$ and let $R = \min( \tilde r, \dist(K, \partial \Omega))$, where $\tilde r$ is the radius introduced in Theorem \ref{thm:Almgren}. Reasoning exactly as in \cite[Corollaries 2.6, 2.7 and 2.8]{tt}, we can show the following.

\begin{lemma}\label{lem: bound Almgren}
On the previous assumptions:
\begin{itemize}
	\item the map $\Omega \to \R$, $x\mapsto N(x,\mf u,0^+)$ is upper semi-continuous;
	\item the set $\Gamma_\mf{u}$ has empty interior. Moreover
\[
	\lim_{r \to 0^+} N(x,\mf{u},r) \geq 1 \qquad \forall x \in \Gamma_\mf{u};
\]
	\item there exists a constant $C>0$ such that
\[
	N(x,\mf{u},r) \leq C \qquad \forall x \in K, 0 < r < R.
\]
\end{itemize}
\end{lemma}

We have
\begin{lemma}\label{lem: morrey fb}
There exists a constant $C>0$ such that
\[
	\frac{1}{r^N} \sum_{i=1}^{d} \int_{B_r(x_0)} |\nabla u_i|^2 \leq C \qquad \forall x_0\in K \cap \Gamma_{\mf{u}}, 0 < r < R.
\]
\end{lemma}
\begin{proof}
This is a direct consequence of Theorem \ref{thm:Almgren} and Lemma \ref{lem: bound Almgren}. Indeed, as the Almgren quotient is bounded from below, we have
\[
	2 \leq e^{Cr^2} (N(x_0,\mf{u}, r) + 1) \implies N(x_0,\mf{u}, r) > 2e^{-Cr^2} - 1
\]
and, moreover,
\[
	\frac{d}{d r} \log \frac{H(x_0,\mf{u},r)}{r^2} = \frac{2}{r}\left(N(x_0,\mf{u},r)-1\right) \geq \frac{4}{r} \left(e^{-Cr^2} - 1\right).
\]
Integrating the previous inequality in $(r, R)$, for a generic $0 < r < R$ we find that there exists yet another constant $C>0$ such that
\[
	\frac{H(x_0,\mf{u},r)}{r^2} \leq C \frac{H(x_0,\mf{u},R)}{R^2} \qquad \text{for all $0 < r < R$.}
\]
We then exploit the boundedness of the Almgren quotient, from which we obtain
\[
	N(x_0,\mf{u}, r) < C \implies \frac{E(x,\mf{u},r) + H(x,\mf{u},r)}{r^2} \leq C \frac{H(x,\mf{u},r)}{r^2}  \leq C \frac{H(x_0,\mf{u},R)}{R^2}.
\]
Let us observe that, since $\mf{u}$ is continuous and $R$ is a fixed positive radius, the last term of the previous inequality is bounded uniformly from above. The conclusion now follows from an application of Poincar\'e inequality, see also Theorem \ref{thm:Almgren}.
\end{proof}

\begin{proof}[Conclusion of the proof of Proposition \ref{prp: lip}]
We are now in a position to conclude the uniform boundedness of the Morrey quotient \eqref{eqn: morrey}, and in turn, the Lipschitz continuity of the functions $\mf{u}$. To do so, we resort once again to a contradiction argument, and we assume that there exists a sequence $(x_n,r_n)$ so that $x_n \in K$ and $r_n > 0$, for which
\[
	\phi(x_n,r_n) = \frac{1}{r_n^N} \sum_{i=1}^{d} \int_{B_{r_n}(x_n)} |\nabla u_i|^2 \to +\infty.
\]
As $\mf{u} \in H^1(\Omega)$, it is easy to see that, necessarily, $r_n \to 0$. Let $x_0 =\lim x_n$. At first, we rule out two initial cases:
\begin{itemize}
	\item $x_0 \not \in \Gamma_\mf{u}$. Indeed, this is the content of of Lemma \ref{lem: morrey fb}, which would otherwise imply $\phi(x_n,r) < C$.
	\item it must be $\rho_n := \dist(x_n, \Gamma_\mf{u}) \to 0$. Otherwise, let $\bar \rho > 0$ be such that $\rho_n > \bar \rho$. For any fixed $n$, there would exists $h \in \{1, \dots, m\}$ such that for all $j \not \in I_h$, $u_j = 0$, while for $i \in I_h$
	\[
		-\Delta u_i = f_i(x,\mf{u}) \qquad \text{in $B_{\bar \rho} (x_n)$.}
	\]
	As a result, by the Calderon-Zygmund inequality (see \cite[Theorem 9.11]{GiTr}), we have the uniform control
	\[
		\| \mf{u} \|_{W^{2,q}(B_{\bar \rho / 2})} \leq C_q \left( \|\mf{u}\|_{L^q(B_{\bar \rho})} + \|\mf{f}\|_{L^q(B_{\bar \rho})}  \right)
	\]
	for a constant $C$ which is independent of $x_n$. Recalling the assumptions on $f_i$ and the boundedness of $\mf{u}$, we see that in the previous estimate we can take any power $1 < q < \infty$: in particular, for $q > N$, by the Sobolev embedding theorem the Morrey quotient $\phi(x_n,r)$ is bounded from above independently of $0 < r <  \bar \rho/ 2$.
\end{itemize}
We can easily exclude another possible behaviour of the sequence $(x_n,r_n)$. Letting $\bar x_n \in \Gamma_\mf{u}$ be any point of the free boundary such that $\rho_n = \dist(x_n, \bar x_n) = 2 \dist(x_n, \Gamma_u)$ we have
\begin{itemize}
	\item $r_n / \rho_n \to 0$, that is, $\rho_n$ can not be comparable with $r_n$. Indeed, if there exists $C > 0$ such that $r_n > C \rho_n$, then
	\[
		\frac{1}{r_n^N} \sum_{i=1}^{d} \int_{B_{r_n}(x_n)} |\nabla u_i|^2 \leq \frac{1}{(C \rho_n)^N} \sum_{i=1}^{d} \int_{B_{4\rho_n}(\bar x_n)} |\nabla u_i|^2 \leq \frac{C}{\rho_n^N} \sum_{i=1}^{d} \int_{B_{2\rho_n}(\bar x_n)} |\nabla u_i|^2.
	\]
	As a consequence, we have reduced this case to the estimate from above on points of the free boundary $\Gamma_\mf{u}$, thus leading to a contradiction.
\end{itemize}
To conclude the proof, we can reason as in \cite[Theorem 8.3, case II]{ctvIndiana} .
\end{proof}

\section{Regularity of the free boundary $\widetilde \Gamma_{\mf{u}}$ for $\mf{u}\in \mathcal{G}(\Omega)$}\label{sec: regularity}

We will divide the proof of the regularity of $\widetilde \Gamma_{\mf u}$ in two subsections: in the next one, we first present some general Boundary Harnack Principles, and then in Subsection \ref{subseq:ProofRegularity} we prove Theorem \ref{thm:regularity_G}. If we assume moreover that $u_i \geq 0$ for every $i$, they we can actually prove regularity results for the whole nodal set $\Gamma_{\mf{u}}$. Since the proof of this case presents very few differences with respect to the general proof, the case of nonnegative components will be treated in Remark \ref{rem:positive}.

\subsection{Boundary Harnack Principles on NTA and Reifenberg flat domains}\label{subseq:Harnack}

Let $\omega$ be a \emph{non-tangencially-accessible} (NTA) domain, a notion introduced in \cite{JerisonKenig}. We start by proving a Boundary Harnack Principle for solutions of
\begin{equation}\label{eq:-Delta_w=lambda w}
-\Delta u=a(x) u,\qquad a\in L^\infty(\omega),
\end{equation}
which will be a straightforward extension of the seminal paper of Jerison and Kenig \cite{JerisonKenig} (see also the book by Kenig \cite{Kenig}).

\begin{lemma}\label{lemma:boundaryHarnack_same_lambda}
Let $\omega$ be an NTA domain, $a\in L^\infty(\omega)$, and $x_0\in \partial \omega$. Then there exist $R_0,C>0$ (depending only on $a(x)$ and the NTA constants) such that for every $0<2r<R_0$ and for every $u,v$ solutions of \eqref{eq:-Delta_w=lambda w} in $\omega\cap B_{2r}(x_0)$ with $u=v=0$ on $\partial \omega\cap B_{2r}(x_0)$, and $u,v> 0$ in $\omega$, then $u/v$ can be extended up to $\partial \omega\cap B_{r}(x_0)$, and
\begin{equation}\label{eq:BHarnack}
C^{-1}\frac{v(y)}{u(y)}\leq \frac{v(x)}{u(x)}\leq C\frac{v(y)}{u(y)} \quad \forall x\in \overline \omega \cap B_r(x_0),\  y\in \omega \cap B_r(x_0).
\end{equation}
Moreover, there exists $\alpha\in (0,1)$ such that 
\[
\frac{v}{u}\text{ is H\"older continuous of order } \alpha \text{ on } \overline{\omega \cap B_r(x_0)}.
\]
More precisely,
\[
\left|\frac{v(x)}{u(x)}-\frac{v(y)}{u(y)}\right|\leq C \frac{v(z)}{u(z)} \frac{|x-y|^\alpha}{r^\alpha},\qquad \forall x,y \in \overline{\omega\cap B_r(x_0)},\ z\in \omega\cap  B_r(x_0).
\]
\end{lemma}

\begin{proof}
 Take $\varphi_0$ a solution of
\[
-\Delta \varphi_0=a(x) \varphi_0 \text{ in } B_{2R_0}(x_0),\qquad \varphi_0>0 \text{ on } B_{2R_0}(x_0)
\]
(which exists for $R_0>0$ sufficiently small, depending on $a(x)$). Then
\[
\div \left(\varphi_0^2\nabla \left(\frac{u}{\varphi_0}\right)\right)=\div\left(\varphi_0^2\nabla \left(\frac{v}{\varphi_0}\right)\right)=0 \qquad \text{ in } B_{2R_0}(x_0)
\]
and we can apply the classical Boundary Harnack Principle for divergence-type operators \cite[Lemma 1.3.7 \& Corollary 1.3.9]{Kenig} to $u/\varphi_0, v/\varphi_0$, which provides the result.
\end{proof}
Now the main focus will be to prove H\"older continuity up to the boundary for quotients of solutions to two problems of type \eqref{eq:-Delta_w=lambda w} with different potentials $a(x),b(x)$. For that, we will need to require extra assumptions for the solutions, and assume that $\omega$ is  a $(\delta,R)$--Reifenberg flat domain (see \cite{KenigToro}, or Proposition \ref{prop:Reifenberg} ahead to check the definition). We shall always take $\delta=\delta(N)>0$ small so that $\omega$ is also an NTA domain (\cite[Theorem 3.1]{KenigToro}). We will show the following.

\begin{proposition}\label{prop:u_i/u_1_alpha_Holder_generalcase}
Let $\omega$ be a $(\delta,R)$--Reifenberg flat domain, $a,b\in L^\infty(\omega)$, $x_0\in \partial \omega$ and $R_0>0$. Take $u,v$ solutions of
\[
-\Delta u=a(x)u,\quad -\Delta v=b(x)v \qquad \text{ in } \omega \cap B_{R_0}(x_0),
\]
\[
u,v>0 \text{ in } \omega\cap B_{R_0(x_0)},\qquad u,v=0 \text{ in } \partial\omega\cap B_{R_0(x_0)},
\]
with $u$ Lipschitz continuous in $\overline{\omega\cap B_{R_0}(x_0)}$. Assume moreover that: given $x_n\in \partial \omega$ with $x_n\to x_0$ and $t_n\to 0^+$, there exists $\rho_n$, $\gamma>0$ and $e\in S^{N-1}$ such that
\begin{equation}\label{eq:assumption_technical}
\frac{\rho_n}{t_n^{1+\eps}}\not \to 0 \text{ as } n\to\infty, \qquad \forall \eps \text{ small}
\end{equation}
and
\[
\frac{u(x_n+t_nx)}{\rho_n}\to \gamma(x\cdot e)^+,\qquad \left|\frac{v(x_n+t_nx)}{\rho_n}\right| \leq C \qquad \text{ uniformly in each compact set.}
\]
Then $v/u$ can be continuously extended up to the boundary of $\omega$, and there exists $C>0$ such that, for $r$ sufficiently small,
\[
\left|\frac{v(x)}{u(x)}-\frac{v(x_0)}{u(x_0)}\right|\leq C r^{\alpha} \qquad \forall x\in \overline{B_{r}(x_0)\cap \omega}.
\]
\end{proposition}

The aim of the remainder of the subsection is to prove this result. The idea is to consider suitable deformations of $u$ so that the resulting functions are either sub or supersolutions of the equation $- \Delta w=b(x) w$ with comparable boundary data with respect to $u$, considering afterwards some $b(x)$-harmonic extensions in view of using Lemma \ref{lemma:boundaryHarnack_same_lambda}. 

We start by deforming $u$ into a subsolution. Take $\eps>0$, and let $g:\R\to \R$ be defined as
\[
g(s)=s+\frac{s^{3-\eps}}{(3-\eps)(2-\eps)}.
\]
Then
\begin{align}
-\Delta (g\circ u )&=-\div (g'(u)\nabla u)=-g'(u)\Delta u-g''(u)|\nabla u|^2 = a(x) g'(u)u-g''(u)|\nabla u|^2\\
			 &=a(x) \left(1+\frac{u^{2-\eps}}{2-\ep}\right)u-u^{1-\eps} |\nabla u|^2=u\left(a(x)-|\nabla u|^2u^{-\eps}+\frac{a(x)}{2-\eps}u^{2-\eps}\right)\\
			 &= u\left(a(x) -\frac{|\nabla u|^2 d^2(x)}{u^2} \frac{u^{2-\eps}}{d^2(x)}+\frac{a(x)}{2-\eps}u^{2-\eps}\right),\label{eq:-Delta g(v)}
\end{align}
in $\omega\cap B_{R_0}(x_0)$, where we denote $d(x):=\dist(x,\partial \omega)$.

\begin{lemma}\label{lemma:properties_of_rho}
Given $x_0\in \partial\omega$, there exists $R_0$ and $C>0$ such that:
\begin{itemize}
\item[(i)] $\displaystyle C^{-1}\leq \frac{|\nabla u (x)|^2d^2(x)}{u^2(x)}\leq C$ for every $x\in B_{R_0}(x_0)\cap \omega$;
\item[(ii)] $\displaystyle \lim_{x\to x_0}\frac{u(x)^{2-\ep}}{d^2(x)}=+\infty$, for every $\ep>0$ small.
\end{itemize}
\end{lemma}

\begin{proof}
(i) The proof goes by contradiction. Suppose there exists $r_n\to 0$ and $x_n\in B_{r_n}(x_0)\cap \omega$ such that
\[
\frac{|\nabla u(x_n)|^2d^2(x_n)}{u^2(x_n)} \qquad \text{ converges either to } 0 \text{ or to } +\infty.
\]
Let $t_n:=d(x_n)\to 0$ and take $x_n'\in \partial\omega$ such that $d(x_n)=|x_n-x_n'|$. Then, by  assumption, there exists $\rho_n$ such that the blowup sequence  
\[
u_{n}(x):=\frac{u(x_n'+t_n x)}{\rho_n},\text{ extended by 0 to } \frac{B_{R_0}(x_0)-x_n'}{t_n},
\]
converges (without loss of generality) to $\bar u=\gamma(x\cdot e)^+$, for some $\gamma>0$, $e\in S^{N-1}$. Observe that
\[
\frac{|\nabla u(x_n)|^2d^2(x_n)}{u^2(x_n)}=\frac{|\nabla u_n(\frac{x_n-x_n'}{t_n})|^2}{u_n^2(\frac{x_n-x_n'}{t_n})}
\]
Now $ \dist(\frac{x_n-x_n'}{t_n},\frac{\partial \Omega-x_n'}{t_n} )=\left|\frac{x_n-x_n'}{t_n}\right|=1$ and $\frac{x_n-x_n'}{t_n}\to \bar x\in \partial B_1(0)$. Since, by elliptic regularity, the convergence $u_{n}\to \bar u$ is $\mathcal{C}^{1,\alpha}$ in the complementary of any strip around $\{x\cdot e=0\}$, then we have
\[
\frac{|\nabla u_n(\frac{x_n-x_n'}{t_n})|^2}{u_n^2(\frac{x_n-x_n'}{t_n})}\to \frac{|\nabla \bar u(\bar x) |^2}{\bar u^2(\bar x)}=\frac{1}{((\bar x\cdot e)^+)^2}\in (0,+\infty),
\]
which is a contradiction.

\medbreak

\noindent (ii) Let $x_n\to x_0$ and $t_n:=|x_n'-x_n|=d(x_n)$, and take the corresponding $\rho_n$ given by the statement of Proposition \ref{prop:u_i/u_1_alpha_Holder_generalcase}. By defining $u_n$ as before, one can check that there exists $C>0$ such that
\[
\frac{u(x_n)}{\rho_n}=u_n\left(\frac{x_n-x_n'}{t_n}\right)\in [1/C,C].
\]
Take $\eps'>0$ so that \eqref{eq:assumption_technical} holds and let $\eps/(2-\eps)>\eps'$. Then
\[
\frac{u(x_n)^{2-\eps}}{t_n^2}=\frac{u(x_n)^{2-\eps}}{\rho_n^{2-\eps}}\left(\frac{\rho_n}{t_n^\frac{2}{2-\eps}}\right)^{2-\eps}\to +\infty \qquad \text{ as } n\to \infty. \qedhere
\]
\end{proof}

A simple consequence of the previous lemma together with \eqref{eq:-Delta g(v)} is the following:

\begin{lemma}\label{lemma:properties_of_g(rho)}
Given $x_0\in \partial\omega$, there exists $R_0>0$ such that
\[
-\Delta (g\circ u) \leq b(x) (g\circ u),\quad g\circ u>0 \qquad  \text{ in } \omega \cap B_{R_0}(x_0)
\]
\end{lemma}

Now take the function $h:\R^+\to \R$ defined by 
\[
h(s)=s-\frac{s^{3-\ep}}{(3-\ep)(2-\ep)}
\]
and observe that $h\circ u>0$ and 
\begin{align*}\label{eq:-Delta h(u_1)}
-\Delta (h\circ u )&=-\div (h'(u)\nabla u)=-h'(u)\Delta u-h''(u)|\nabla u|^2 = a(x) h'(u)u-h''(u)|\nabla u|^2\\
			 &= u\left(a(x) +\frac{|\nabla u|^2 d^2(x)}{u^2} \frac{u^{2-\ep}}{d^2(x)}-\frac{a(x)}{(2-\ep)}u^{2-\ep}\right)\geq b(x)(h\circ u)
\end{align*}
in $B_r(x_0)\cap \omega$, for sufficiently small $r>0$ (again by Lemma \ref{lemma:properties_of_rho}).

Let $\bar u_r$ and $\tilde u_r$ the $b(x)$--harmonic extensions in $B_r(x_0)\cap \omega$ of $g\circ u$ and $h\circ u$ respectively, that is:
\[
\begin{cases}
-\Delta \bar u_r=b(x)\bar u_r & \text{ in } B_r(x_0)\cap \omega\\
\bar u_r=g\circ u &\text{ on } \partial (B_r(x_0)\cap \omega)
\end{cases},
\qquad
\begin{cases}
-\Delta \tilde u_r=b(x)\tilde u_r & \text{ in } B_r(x_0)\cap \omega\\
\tilde u_r=h\circ u &\text{ on } \partial (B_r(x_0)\cap \omega)
\end{cases}.
\]
By the comparison principle and the definitions of $g$ and $h$, one has
\[
u\leq g\circ u\leq \bar u_r,\quad \text{ and } \quad \tilde u_r\leq h\circ u \leq u.
\]
Moreover, on $\partial (B_r(x_0)\cap\Omega)$, by using the fact that $u$ is Lipschitz continuous,
\[
g \circ u = h\circ u \left[1+\frac{\frac{2u^{2-\ep}}{(3-\ep)(2-\ep)}  }{1-\frac{u^{2-\ep}}{(3-\ep)(2-\ep)}} \right]\leq h\circ u (1+Cr^{2-\ep}).
\]
Thus, for $C>0$  independent of $r>0$,
\[
\bar u_r\leq \tilde u_r(1+C r^{2-\ep}),\quad \text{whence} \quad u\leq \bar u_r\leq \tilde u_r(1+C r^{2-\ep})\leq u(1+Cr^{2-\ep}),
\]
and in particular
\[
1\leq\frac{\bar u_r}{u}\leq (1+C r^{2-\ep}) \quad \text{ in } B_r(x_0)\cap \omega.
\]

\begin{lemma}
Under the previous notations, there exist $0<\delta \ll 1$, $0<\alpha<1$, $C>0$, such that
\[
\left|\frac{v(x)}{\bar u_r(x)}-\frac{v(y)}{\bar u_r(y)} \right|\leq Cr^{(1-\delta)\alpha/\delta},\qquad \forall x,y \in \overline{\omega\cap B_{r^{1/\delta}}(x_0)}.
\]
\end{lemma}

\begin{proof}
By applying Lemma \ref{lemma:boundaryHarnack_same_lambda} to $v$ and $\bar u_r$, we deduce the existence of $C>0$ (independent of $r$) such that
\[
\left|\frac{v(x)}{\bar u_r(x)}-\frac{v(y)}{\bar u_r(y)} \right|\leq C \frac{v(\xi)}{\bar u_r(\xi)}\frac{|x-y|^\alpha}{r^\alpha},\qquad \forall x,y \in \overline{\omega\cap B_{r}(x_0)},\ z\in \omega \cap B_r(x_0).
\]
Reasoning as in the proof of Lemma \ref{lemma:properties_of_rho}, by choosing $\xi=\xi_r\in \partial B_{r/2}(x_0)\cap \omega$ such that $\dist(\xi,\partial \Omega)\geq r \ep$ (which exists since $\omega$ is Reifenberg flat) one proves that the quotient $\frac{v(\xi)}{\bar u_r(\xi)}$ is bounded. Take $\delta>0$ small. Then we conclude that
\[
\left|\frac{u_i(x)}{u_r(x)}-\frac{u_i(y)}{u_r(y)} \right|\leq C' \frac{|x-y|^\alpha}{r^\alpha}\leq C'\frac{r^{\alpha/\delta}}{r^\alpha}=C'r^{(1-\delta)\alpha/\delta},\qquad \forall x,y \in \overline{\omega\cap B_{r^{1/\delta}}(x_0)}. \qedhere
\]
\end{proof}

\begin{proof}[Proof of Proposition \ref{prop:u_i/u_1_alpha_Holder_generalcase}]
Using the decomposition
\[
\frac{v}{u}=\frac{v}{ \bar u_r} \frac{\bar u_r}{u} ,
\]
we have
\begin{align*}
\left|\frac{v(x)}{u(x)}-\frac{v(x_0)}{u(x_0)}\right| &\leq \left| \frac{v(x)}{\bar u_r(x)}-\frac{v(x_0)}{\bar u_r(x_0)}\right| \left| \frac{\bar u_r(x)}{u(x)}\right| + \left| \frac{\bar u_r(x)}{u(x)}-\frac{\bar u_r(x_0)}{u(x_0)}\right| \left| \frac{v(x_0)}{\bar u_r(x_0)}\right|\\
&\leq C' r^{(1-\delta)\alpha/\delta}(1+C r^{2-\ep})+C'' r^{2-\ep}\leq \kappa r^{(1-\delta)\alpha/\delta}
\end{align*}
for every $x\in \overline{\Omega\cap B_{r^{1/\delta}}(x_0)}$, and the result follows.
\end{proof}

\subsection{Conclusion of the proof of regularity results}\label{subseq:ProofRegularity}

After having established some Boundary Harnack Principles in the previous subsection, the proof of Theorem \ref{thm:regularity_G} will mostly follow the papers \cite{rtt,tt}. In \cite{tt}, the case $\#I_h=1$ is treated, while in \cite{rtt} although the segregation is between groups, only the case $f_i(x,\mf{u})=\lambda_i \mf{u}$  is handled. We will prove Theorem \ref{thm:regularity_G} highlighting only the strategy as well as  the main differences with respect to \cite{rtt,tt}.

We observe that Theorem \ref{thm:Almgren} and Lemma \ref{lem: bound Almgren} hold, as they are stated, also for functions $\mf{u} \in \Gcal (\Omega)$: the proofs proceed as in the quoted statements. Moreover, we have that
\begin{equation}\label{wlog}\tag{A}
\text{ For every } x_0\in \widetilde \Gamma_{\mf u},\ \delta>0, \text{ there exists } k\neq h \text{ such that } \sum_{i\in I_h} |u_i|, \sum_{j\in I_k} |u_j|\not\equiv 0 \text{ in } B_\delta(x_0).
\end{equation}

For $x_0\in \Omega$, let $(x_n)$, $x_n\to x_0$ and $t_n\to 0^+$, we define the blow-up sequence $\mf{u}_n:=(u_{1,n},\ldots, u_{d,n})$, as
\[
	u_{i,n}(x):=\frac{u_i(x_n+t_n x)}{\sqrt{H(x_n,\mf{u},t_n)}},\qquad x\in \Omega_n:=\frac{\Omega-x_n}{t_n}.
\]
Observe that 
\[
-\Delta u_{i,n}=f_{i,n}(x,u_{i,n})-\Mcal_{i,n}
\]
with
\[
f_{i,n}(x,s)=\frac{t_n^2}{\sqrt{H(x_n,\mf{u},t_n)}}f_i(x_n+t_nx,\sqrt{H(x_n,\mf{u},t_n)}s)\] 
and
\[
\Mcal_{i,n}(E)=\frac{1}{t_n^{N-2}\sqrt{H(x_n,\mf{u},t_n)}}\Mcal(x_n+t_n E).
\]
Reasoning as in Theorem 4.1, Corollary 4.3 and Corollary 4.5 in \cite{rtt}, one proves the following.

\begin{theorem}
Within the previous framework, given $x_n\to x_0\in \Omega$ and $t_n\to 0^+$, there exists $\bar{\mf{u}}$ with $\bar u_i\cdot \bar u_j\equiv 0$ whenever $i\in I_h$, $j\in I_k$ with $h\neq k$ and measures $\Mcal_i\in \Mcal_{loc}(\R^N)$ such that, up to a subsequence,
\begin{align*}
&\mf{u}_n\to \bar{\mf{u}}\qquad  \text{ in } \mathcal{C}^{0,\alpha}_{loc}\cap H^1_{loc}(\R^N),\ \forall 0<\alpha<1\\
&\Mcal_{i,n}\rightharpoonup \bar{\mathcal{M}}_i\qquad  \text{ weakly--}\star \Mcal_{loc}(\R^N). 
\end{align*}
Moreover, $-\Delta \bar u_i=-\bar \Mcal_i$, the measures $\bar \Mcal_i$ are concentrated on $\Gamma_{\bar{\mf{u}}}$, and it holds
\begin{equation}\label{eq:pohozaev_for_bar_uv2}
(2-N) \sum_{i=1}^d \int_{B_r(x)} |\nabla\bar u_i|^2=\sum_{i=1}^d \int_{\partial B_r(x)} r(2(\partial_{n} \bar u_i)^2-|\nabla \bar u_i|^2)\quad \forall x\in \R^N,\ r>0.
\end{equation}
In particular, $\bar{\mf{u}}\in \Gcal_{loc}(\R^N)$.

Finally, if either $x_n\equiv x_0$, or $x_n\in \Gamma_{\mf u}$ and $N(x_0,\mf{u},0^+)=1$, then
\[
\bar u_i=r^\alpha g_i(\theta),\qquad \text{with }\alpha=N(0,\bar u,r).
\]
\end{theorem}

Given $y\in \Omega$, from now we define the set of all possible blowup limits at $y$ by
\[
\Bcal\Ucal_{y}=\left\{(\bar u,\bar v):  \begin{array}{l}\exists\; x_n\to x_0,\ t_n\to 0 \text{ such that, for every $i$,}\\[4pt]
					 \displaystyle u_{i,n}:=\frac{u_i(x_n+t_n\cdot)}{\sqrt{H(x_n,\mf{u}, t_n)}}\to \bar u_i  	 \text{ strongly in }  H^1_\textrm{loc}(\R^N)\cap \mathcal{C}^{0,\alpha}_\textrm{loc}(\R^N)
					 		 \end{array}\right\}
\]

With the latter compactness result at hand, one can prove a gap condition of the values of $N(x_0,\mf{u},0^+)$, and to characterise completely the blow up limits at points where $N(x_0,\mf{u},0^+)$.

\begin{proposition}\label{prop:blowup_halfspace} Let $\mf{u}\in \Gcal(\Omega)$ and $x_0\in \Gamma_{\mf u}$. Then either
\begin{equation}\label{eq:gap}
N(x_0,\mf{u},0^+)=1 \qquad \text{ or } N(x_0,\mf{u},0^+)\geq 3/2.
\end{equation}
Moreover, if $x_0\in \widetilde \Gamma_{\mf u}$ with $N(x_0,\mf u,0^+)=1$ and $\bar{\mf{u}}\in \Bcal\Ucal_{x_0}$, then there exists $\nu\in S^{N-1}$, $k\neq h$ and $\alpha_i,\beta_j\in \R$ for $i\in I_h, j\in I_k$ such that
\[
\bar u_i=\alpha_i (x\cdot \nu)^+ \text{ for } i\in I_h,\qquad \bar u_j=\beta_j (x\cdot \nu)^+ \text{ for } j\in I_k
\]
Moreover, we have the following compatibility condition
\[
\sum_{i\in I_h} \alpha_i^2=\sum_{j\in I_j} \beta_j^2\neq 0, \quad \text{ so that } \sum_{i\in I_h} |\nabla \bar u_i|^2=\sum_{j\in I_h} |\nabla \bar u_j|^2 \text{ on } \{x\cdot \nu=0\}.
\]
\end{proposition}
\begin{proof} (Sketch) Repeating the proof in \cite[Proposition 4.7]{rtt}, one proves \eqref{eq:gap}. Observe that the fact of having nontrivial grouping combined with eventually sign-changing solutions which are not minimisers, makes the proof more delicate than the one appearing in \cite[Lemma 4.1]{CaffLin} and \cite[Proposition 3.7]{tt}. 

Moreover, one sees that if $N(x_0,\mf{u},0^+)=1$  and $\bar{\mf{u}}\in \Bcal\Ucal_{x_0}$, then $\Gamma_{\bar{\mf{u}}}$ is a vector space having dimension at most $N-1$, being exactly $N-1$ except in the possible case where all but one group of components is trivial. However, this latter case is excluded for $x_0\in \widetilde \Gamma_{\mf{u}}$ by the Clean Up Lemma \cite[Proposition 4.15]{rtt} combined with condition \eqref{wlog}. Thus the situation is as follows at such points: $\Gamma_{\bar{\mf u}}$ has exactly two connected components, let us denote them by $A$ and $B$. In such a case, one shows that there exists $h\neq k$ such that
\[
\text{ for each } i \in I_h,\qquad \text{ either }|\bar u_i|>0 \text{ in } A \ \text{ and }\  \bar u_i=0 \text{ on } \partial A,\qquad  \text{or }\quad  \bar u_i\equiv 0.
\]
and
\[
\text{ for each } j \in I_k,\qquad \text{ either }|\bar u_j|>0 \text{ in } B \ \text{ and }\  \bar u_i=0 \text{ on } \partial B,\qquad  \text{or }\quad  \bar u_j\equiv 0.
\]
(where we have also taken in consideration assumption \eqref{wlog}). Then all functions are first eigenfunctions on the corresponding support, and if $I_h=\{h_1,\ldots, h_l\}$, $I_k=\{k_1,\ldots, k_{\tilde l}\}$  there exists $\alpha_i, \beta_j \in \R$ with $i\in I_h$, $j\in I_k$ such that
\[
\bar u_{h_i}=\alpha_i u_{h_1},\qquad \bar u_{k_j}=\beta_j u_{k_1}.
\]
Now since $\bar{\mf u}\in \mathcal{G}_{loc}(\R^N)$, the new functions
\[
\tilde u:=\sqrt{\sum_{i\in I_h} \alpha_i^2}|\bar u_{h_1}|,\qquad \tilde v:=\sqrt{\sum_{i\in I_k} \beta_i^2}|\bar u_{k_1}|
\]
are such that $(\tilde u,\tilde v)$ belong to $\Gcal_{loc}(\R^N)$ in the case $d=2$ (the case of exactly two segregated species). Thus by \cite[Lemma 6.1]{tt} we have that $\Gamma_{\bar{\mf u}}=\{x\cdot \nu=0\}$ for some $\nu\in S^{N-1}$, and $\bar u_{h_1}=\gamma (x\cdot \nu)^+$, $\bar u_{k_1}=\gamma (x\cdot \nu)^-$, with $\gamma>0$. By using \eqref{eq:pohozaev_for_bar_uv2} and reasoning exactly as in point 3. of the proof of Theorem 4.16 in \cite{rtt}, we get
$\sum_{i\in I_h} \alpha_i^2=\sum_{j\in I_k} \beta_j^2$.
\end{proof}

Following the literature, we now define the regular and singular sets as
\begin{align*}
&\Rcal_{\mf u}=\{x\in \widetilde \Gamma_{\mf u}:\ N(x_0,\mf u,0^+)=1\},\\
&S_{\mf u}=\{x\in \widetilde \Gamma_{\mf u}: N(x_0,\mf u,0^+)> 1\}=\{x\in \Gamma_{\mf u}: N(x_0,\mf u,0^+)\geq  3/2\}.
\end{align*}
We can apply the Federer's Reduction Principle (see for instance Appendix A in \cite{simon}), proving already part of Theorem \ref{thm:regularity_G}. 

\begin{theorem}\label{thm:hausdorff_measures_of_nodalsets}
For any $N\geq 2$ we have that:
\begin{itemize}
\item[1.] $\Hh_\textrm{dim}(\Gamma_{\mf u})\leq N-1$;
\item[2.] $\Hh_\textrm{dim}(S_{\mf u})\leq N-2$. Moreover, if $N=2$, for any compact $\tilde \Omega \Subset \Omega$ the set $S_{\mf u}\cap \tilde \Omega$ is finite.
\end{itemize}
\end{theorem}
\begin{proof}
For the complete details, see \cite[Theorem 4.5 \& Remark 4.7]{tt}.
\end{proof}

Moreover, the information for the blowups in $\Bcal\Ucal_{x_0}$ with $x_0\in \Rcal_{\mf u}$, allows us to reason as in \cite[Lemma 3.5 \& Proposition 5.4]{tt}, proving the following ($d_{\rm Hausd}$ denotes the Hausdorff distance).

\begin{proposition}\label{prop:Reifenberg}
Fix $x_0\in \Rcal_{\mf u}$. Then there exists $R_0>0$ such that the set $B_{R_0}(x_0)\setminus \Gamma_{\mf u}$ has exactly two connected components $\Omega_1,\Omega_2$, which are $(\delta,R)$--Reifenberg flat for every small $\delta>0$ and some $R=R(\delta)$. More precisely: for every $\delta>0$ there exists $R>0$ such that whenever $x\in \Gamma_{\mf u}\cap B_R(x_0)$, $0<r<R$ there exists an hyperplane $H=H_{x,r}$ containing $x$ satisfying
\begin{enumerate}
\item[i)]  $\text{d}_{\rm Hausd}(\Gamma_{\mf u}\cap B_r(x),H\cap B_r(x))\leq \delta r$
\item[ii)] there exists a unitary vector $\nu=\nu_{x,r}$ orthogonal to $H_{x,r}$ such that
\[
\{y+t\nu\in B_r(x):\ y\in H,\ t\geq \delta r\}\subset \Omega_1,\qquad \{y-t\nu\in B_r(x):\ y\in H,\ t\geq \delta r\}\subset \Omega_2.
\]
\end{enumerate}
\end{proposition}

In view of proving Theorem \ref{thm:regularity_G}, let us now focus in the local regularity of $\Rcal_{\mf u}$. Fix $x_0$ in such set. Then, from the previous proposition, we get the existence of $R_0>0$, sets $\Omega_1,\Omega_2$, and $k\neq h$ such that
\[
\sum_{i\in I_h} |u_i|>0 \text{ in } \Omega_1,\qquad \sum_{j\in I_k} |u_j|>0 \text{ in } \Omega_2.
\]
Let $1 \leq h_1, k_1 \leq d$ and $l, \tilde l$ be such that $I_h=\{h_1,\ldots, h_1+l=:h_l\}$ and $I_k=\{k_1,\ldots, k_1+\tilde l=:k_{\tilde l}\}$ and define
\[
\mf{u}^h:=(u_{h_1},\ldots, u_{h_l}),\qquad \mf{u}^k:=(u_{k_1},\ldots, u_{k_{\tilde l}}).
\]

Let us check that in a neighbourhood of $x_0$ at least one component of $\mf{u}^h$ and of $u_{h_l}$ does not change sign.

\begin{lemma}
There exists $R>0$, $h_i\in \{h_1,\ldots, h_l\}$ and $k_j\in \{k_1,\ldots, k_{\tilde l}\}$ such that
\[
\text{ either } \quad u_{h_i}>0 \quad \text{ or } \quad u_{h_i}<0 \quad \Omega_1\cap B_r(x_0),
\]
and
\[
\text{ either } \quad u_{k_i}>0 \quad \text{ or } \quad u_{k_i}<0 \quad \Omega_2\cap B_r(x_0).
\]
\end{lemma}

\begin{proof}
From Proposition \ref{prop:blowup_halfspace}, we have, for any given $t_n\to 0$,
\[
\frac{u_{h_i}(x_0+t_nx)}{\sqrt{H(x_0,\mf u,t_n)}}\to \bar u\not\equiv 0
\]
for some $h_i\in I_h$ (the index eventually depending on $\{t_n\}$). Assume without loss of generality that $\bar u\geq 0$.

Define by $w_n,z_n$ the $\frac{f(u_{h_i})}{u_{h_i}}$--harmonic extensions of $u_{h_i}^+$ and $u^-_{h_i}$ on $B_{2t_n}(x_0)\cap \Omega_1$, namely:
\[
-\Delta w_n=\frac{f(u_{h_i})}{u_{h_i}} w_n \quad -\Delta z_n=\frac{f(u_{h_i})}{u_{h_i}} z_n \quad \text{ in } B_{2t_n}(x_0)\cap \Omega,
\]
\[
w_n=u_{h_i}^+,\quad z_n=u_{h_i}^- \qquad \text{ on } \partial (B_{2t_n}(x_0)\cap \Omega).
\]
and observe that $u_{h_i}=w_n-z_n$ in $B_{2 t_n}(x_0)\cap \Omega_1$. Let $\widetilde w_n$, $\widetilde z_n$ denote the blowups
\[
\widetilde w_n=\frac{w_n(x_0+t_n x)}{\rho_n}\ \text{ and }\ \widetilde z_n=\frac{z_n(x_0+t_n x)}{\rho_n}, \qquad \text{defined in $B_2(0)\cap\left( \frac{\Omega_1-x_0}{t_n}\right)$}. 
\]
At the limit, we find a harmonic equation in a half sphere (since $\Omega_1$ is a Reifenberg flat domain), and the boundary data converges to $\bar u$ and $0$ respectively. Hence we have
\[
\widetilde w_n\to \bar u\geq 0,\qquad \widetilde z_n\to 0.
\]
Thus there exists $\bar y\in \partial B_{1/2}(0)\cap \left(\frac{\Omega_1-x_0}{t_n}\right)$ such that
\[
\frac{z_n(x_0+t_n\bar y)}{w_n(x_0+t_n\bar y)}=\frac{\widetilde z_n(\bar y)}{\widetilde w_n(\bar y)}<\frac{1}{C}
\]
for $n$ large, where $C$ is the constant appearing in \eqref{eq:BHarnack}. Then by this very same lemma applied to $z_n,w_n$, we have 
\[
\frac{z_n(x_0+t_n x)}{w_n(x_0+t_n x)}\leq C\frac{\widetilde z_n(\bar y)}{\widetilde w_n(\bar y)}<1 \qquad \forall x\in B_1(0)\cap \left(\frac{\Omega-x_0}{t_n}\right),
\]
and so $u_{h_i}=w_n-z_n>0$ in $B_{t_n}(x_0)\cap \Omega_1$ for sufficiently large $n$. The proof for $u_{k_i}$ is analogous.
\end{proof}

Assume, without loss of generality, that $u_{h_1}>0$ in $\Omega_1$ and $u_{k_1}>0$ in $\Omega_2$. 

\begin{lemma}\label{lemma:Holder_different_a}
There exists $C>0$ such that, for $r$ sufficiently small
\[
\left|\frac{u_{h_1+i}(x)}{u_{h_1}(x)}-\frac{u_{h_1+i}(x_0)}{u_{h_1}(x_0)}\right|\leq C r^{\alpha} \qquad \forall x\in \overline{B_{r}(x_0)\cap \Omega_1},\ i=2,\ldots, l
\]
and
\[
\left|\frac{u_{k_1+j}(x)}{u_{k_1}(x)}-\frac{u_{k_1+i}(x_0)}{u_{k_1}(x_0)}\right|\leq C r^{\alpha} \qquad \forall x\in \overline{B_{r}(x_0)\cap \Omega_2},\ j=2,\ldots, \tilde l.
\]
\end{lemma}

\begin{proof}
We prove that, given $x_n\in \Rcal_{\mf u}$ with $x_n\to x_0\in \Rcal_{\mf u}$, and $t_n\to 0^+$,
\begin{enumerate}
\item For every $\eps>0$ small, 
\[
\frac{H(x_n,\mf u,t_n)}{t_n^{2+\eps}}\not \to 0 \quad \text{ as } n\to \infty;
\]
\item we have
\[
\frac{u_{h_1}(x_n+t_n x)}{\sqrt{H(x_n,\mf u,t_n)}}\not \to 0\ \quad \text{ as } n\to \infty;
\]
\end{enumerate}
The first point is a consequence of $N(x',\mf u,0^+)=1$ for $x'\in \Rcal_{\mf u}$. In fact, for every $\eps>0$ small there exists $\bar r>0$ such that for $r\leq \bar r$
\[
N(x',\mf u,r)\leq 1+\ep' \qquad x'\in B_{\delta}(x_0)\cap \Gamma_{\mf u}.
\]
Thus we deduce from Theorem \ref{thm:Almgren} that, for some $C>0$,
\[
C \leq \frac{H(x_n,\mf u,t_n)}{t_n^{2(1+\ep)}}.
\]

As for the second point, take $U(x):=\sum_{i\in I_h} |u_i|$, which satisfies $-\Delta U\leq \lambda U$ in $\Omega_1$. Observe that 
\[
U_n(x)=\frac{U(x_n+t_nx)}{\sqrt{H(x_n,\mf u,t_n)}}\to \gamma(x\cdot \nu)^2\not\equiv 0.
\]
Thus we can reason as in the proof of Lemma \ref{lemma:properties_of_g(rho)} and conclude that for some $R_0>0$ small enough,
\[
-\Delta (g\circ U)\leq -\left\|\frac{f_{h_1}(x,\mf u)}{u_{h_1}}\right\|_\infty (g\circ U) \text{ in } \Omega_1\cap B_{R_0}(x_0).
\]
For sufficiently small $r>0$, let us define $\tilde U_r$ as the $\frac{f_{h_1}(x,\mf u)}{u_{h_1}}$--harmonic extension of $g\circ U$ in $B_r(x_0)\cap \Omega_1$, namely
\[
-\Delta \tilde U_r=\frac{f_{h_1}(x,\mf u)}{u_{h_1}} \tilde U_r \text{ in } B_r(x_0)\cap \Omega_1,\qquad \tilde U_r=g\circ U\geq 0 \text{ on } \partial (B_r(x_0)\cap \Omega_1).
\]
By the comparison principle, for $r>0$ small,
\[
U\leq g\circ U\leq \tilde U_r \qquad \text{ in } B_r(x_0)\cap \Omega_1.
\]
Thus, by Lemma \ref{lemma:boundaryHarnack_same_lambda}, we have that, for any $y\in B_{r/2}(x_0)\cap \Omega_1$ fixed,
\[
C_1:=C^{-1}\frac{u_{h_1}(y)}{\tilde U_r(y)}\leq \frac{u_{h_1}(x)}{\tilde U_r(x)} \qquad \forall x\in B_{r/2}(x_0)\cap \overline \Omega_1.
\]
Thus we obtain the sought lower bound
\[
\frac{u_{h_1}(x_n+t_n x)}{\sqrt{H(x_n,\mf u,t_n)}}\geq C_1 \frac{U(x_n+t_n x)}{\sqrt{H(x_n,\mf u,t_n)}}\not \to 0.
\]
Now if  $u_{h_1+i}$ is signed, we apply directly Proposition \ref{prop:u_i/u_1_alpha_Holder_generalcase}. If instead changes sign, we apply this proposition to the $\frac{f(u_{h_1+i})}{u_{h_1+i}}$--harmonic extensions of $u_{h_1+i}^+$ and $u_{h_1+i}^-$ on $B_r(x_0)\cap \Omega_1$, for sufficiently small $r>0$.
\end{proof}

\begin{theorem}\label{thm:regularity}
The map
\[
|\mf{u}^h(x)|-|\mf{u}^k(x)|=\sqrt{\sum_{i\in I_h} u_i^2(x)}-\sqrt{\sum_{j\in I_k} u_j^2(x)}
\]
is differentiable at each $x_0\in \Rcal_{\mf u}$ with 
\[
\nabla(|\mf{u}^h|-|\mf{u}^k|)(x_0)=:\nu(x_0)\neq 0
\]
where $x_0\mapsto \nu(x_0)$ is $\alpha$--H\"older continuous. In particular, $\Rcal_{\mf u}$ is locally a $\mathcal{C}^{1,\alpha}$--hypersurface, for some $\alpha\in (0,1)$.
\end{theorem}
\begin{proof} (Sketch)
1. For $x\in \Omega_1$, let
\[
\Ucal^h(x)=\frac{\mf{u}^h(x)}{|\mf{u}^h(x)|}=\frac{(u_{h_1}(x),\ldots,  u_{h_1+l}(x))}{\sqrt{u_{h_1}^2(x)+\ldots +u_{h_1+l}^2(x)}}
\]
and, for $x\in \Omega_2$,
\[
\Ucal^k(x)=\frac{\mf{u}^k(x)}{|\mf{u}^k(x)|}=\frac{(u_{k_1}(x),\ldots,  u_{k_1+\tilde l}(x))}{\sqrt{u_{k_1}^2(x)+\ldots +u_{k_1+\tilde l}^2(x)}}.
\]
Since we can rewrite
\[
\Ucal^h=\frac{\left(1,\frac{u_{h_1+1}}{u_{h_1}},\ldots,  \frac{u_{h_1+l}}{u_{h_1}}\right)}{\sqrt{1+\left(\frac{u_{h_1+1}}{u_{h_1}}\right)^2+\ldots +\left(\frac{u_{h_1+l}}{u_{h_1}}\right)^2}},\qquad \Ucal^k=\frac{\left(1,\frac{u_{k_1+1}}{u_{k_1}},\ldots,  \frac{u_{k_1+\tilde l}}{u_{k_1}}\right)}{\sqrt{1+\left(\frac{u_{k_1+1}}{u_{k_1}}\right)^2+\ldots +\left(\frac{u_{k_1+\tilde l}}{u_{k_1}}\right)^2}}
\]
then, applying Lemma \ref{lemma:Holder_different_a}, we deduce that
\[
|\Ucal^h(x)-\Ucal^h(x_0)|\leq C r^\alpha,\quad |\Ucal^k(x)-\Ucal^k(x_0)|\leq C r^\alpha \qquad \forall x\in B_r(x_0),\ r \text{ small}.
\]

\medbreak

2. Let us consider
\[
\mf{u}^h_{x_0}(x)=\Ucal^h(x_0)\cdot \mf{u}^h(x) \text{ for } x\in \Omega_1,\qquad \mf{u}^k_{x_0}(x)=\Ucal^k(x_0)\cdot \mf{u}^k(x) \text{ for } x\in \Omega_2,
\]
which satisfy
\[
-\Delta \mf{u}^h_{x_0}=\sum_{i\in I_h} \Ucal^h_i(x_0) f_i(x,\mf u)-\Mcal^h_{x_0},\quad -\Delta \mf{u}^k_{x_0}=\sum_{j\in I_k} \Ucal^h_j(x_0) f_j(x,\mf u)-\Mcal^k_{x_0}
\]
in $B_r(x_0)$, where $\Mcal^h_{x_0},\Mcal_{x_0}^k$ are nonnegative Radon measures concentrated on $\Gamma_{\mf u}$. Taking $\psi_{x_0,r}$ as the solution of
\[
\begin{cases}
-\Delta \psi_{x_0,r}=\sum_{i\in I_h} \Ucal^h_i(x_0) f_i(x,\mf u)-\sum_{j\in I_k} \Ucal^h_j(x_0) f_j(x,\mf u) & \text{ in } B_r(x_0)\\
\psi_{x_0,r}=\mf{u}^h_{x_0}-\mf{u}^k_{x_0} & \text{ on } \partial B_r(x_0)
\end{cases}
\]
and reasoning exactly as in \cite[Proposition 4.24 \& Lemma 4.26]{rtt}, we obtain the existence of
\[
\nu(x_0):=\lim_{r\to 0}\nabla \psi_{x_0,r}(x_0)\neq 0
\]
and, moreover, the function $\nu: \Gamma_{\mf u}\to \R^N$, $x_0\mapsto \nu(x_0)$ is H\"older continuous. Then Theorem 4.27 in \cite{rtt} provides the final conclusion.
\end{proof}

\begin{proof}[Conclusion of the proof of Theorem \ref{thm:regularity_G}]
Taking in consideration Theorem \ref{thm:hausdorff_measures_of_nodalsets} and Theorem \ref{thm:regularity}, we see that the only thing left to prove are conditions \eqref{eq:reflectionlaw} and \eqref{eq:vanishingofgradient}. 

With respect to the first one, we fix $x_0\in \Rcal_{\mf u}$. Let su observe first of all that, given $x\in \Omega_1$ and $d(x):=d(x,\Gamma_{\mf u})$,
\[
\Ucal^h(x)=\frac{\left(\frac{u_{h_1}(x)}{d(x)},\ldots,  \frac{u_{h_1+l}(x)}{d(x)}\right)}{\sqrt{\left(\frac{u_{h_1}(x)}{d(x)}\right)^2+\ldots +\left(\frac{u_{h_1+l}(x)}{d(x)}\right)^2}}\to -\frac{\partial_\nu \mf{u}^h(x_0)}{|\partial_\nu \mf{u}^h(x_0)|}\qquad \text{ as $x\to x_0$.}
\]
Thus
\[
\nabla \left(\sum_{i\in I_h} u_i^2(x)\right)^{1/2}=\sum_{i\in I_h} u_i \nabla u_i(x) \left(\sum_{i\in I_h} u_i^2(x)\right)^{-1/2}\to |\nabla \mf{u}^h(x_0)| \qquad \text{ as } x\to x_0.
\]
Likewise, we can show that
\[
\nabla \left(\sum_{j\in I_k} u_j^2(x)\right)^{1/2}\to |\nabla \mf{u}^h(x_0)| \qquad \text{ as } x\to x_0,
\]
whence \eqref{eq:reflectionlaw} is a direct consequence of the fact that $|\mf{u}^h|-|\mf{u}^k|$ is differentiable at $x_0$.

As for \eqref{eq:vanishingofgradient}, given $x_0\in S_{\mf u}$,  combining the fact that $N(x,\mf u,0^+)\geq 3/2$ for every $x\in S_{\mf u}$ with Theorem \ref{thm:Almgren} yields
\[
H(x,\mf u,0^+)\leq C r^{3} \qquad \forall x\in S_{\mf u}\cap B_\delta(x_0)
\]
(for $C$ independent from $x$). Using Theorem \ref{thm:Almgren} and the assumptions on $f_i$, it is straightforward to show that
\[
\frac{1}{r^N}\int_{B_{r}(x)} |\nabla \mf{u}^h|^2\leq Cr \qquad \forall x\in S_{\mf u}\cap B_\delta(x_0),\ r\leq \bar r
\]
which allows to arrive at the desired conclusion.
\end{proof}

\begin{remark}\label{rem:positive} 
When $u\in \mathcal{G}(\Omega)$ has nonnegative components, we can replace in the previous considerations $\widetilde \Gamma_{\mf{u}}$ by $\Gamma_{\mf{u}}$. The only difference is that, in such case, we can no longer assume condition \eqref{wlog}. However, this condition was only needed for the proof of Proposition \ref{prop:blowup_halfspace}, namely to prove that if $x_0\in \widetilde \Gamma_{\mf{u}}$ and $N(x_0,\mf{u},0^+)=1$, then $\Gamma_{\bar{ \mf{u}}}$ is a hyperplane. The proof now goes as follows: always following \cite{rtt}, if $x_0\in \Gamma_{\mf{u}}$, $N(x_0,\mf{u},0^+)=1$  and $\bar{\mf{u}}\in \Bcal\Ucal_{x_0}$, then $\Gamma_{\bar{\mf{u}}}$ is a vector space having dimension at most $N-1$, being exactly $N-1$ except in the possible case where all but one group of components is trivial. However, inspecting the proof of \cite[Proposition 4.7]{rtt}, we see that in case all groups are trivial except one, then all nonzero components of the blowup limit must be harmonic in $\R^N$, thus sign-changing. Since $u_i\geq 0$, we get a contradiction. Thus $\Gamma_{\bar{\mf{u}}}$ is always a hyperplane. Notice that this argument fails if $\mf{u}$ has sign-chasing components, as shown by the counterexample $\bar u_1(x) = x_1$, $\bar u_2(x) = x_2$, $\bar u_i \ge 0$ for $i \ge 3$, with $\bar u_1$ and $\bar u_2$ in the same group.
\end{remark}

\appendix

\section{Liouville-type theorems}\label{sec: appendix}

In this appendix we collect all the necessary Liouville theorems that are needed along the paper. Almost all of them had already been proven in previous papers, and for those we give the precise references.

\begin{lemma}\label{lem:liouville2}
Let $u,v\in H^1_{\rm loc}(\R^N)\cap C(\R^N)$ be nonnegative functions satisfying $u\cdot v\equiv 0$ and
\[
-\Delta u\leq 0,\quad -\Delta v\leq 0 \qquad \text{ in } \R^N.
\]
If
\[
\mathop{\sup_{x,y\in \R^N}}_{x\neq y} \frac{|u(x)-u(y)|}{|x-y|^\alpha}<\infty\quad \text{ and } \quad \mathop{\sup_{x,y\in \R^N}}_{x\neq y} \frac{|v(x)-v(y)|}{|x-y|^\alpha}<\infty,
\]
then either $u\equiv 0$ or $v\equiv 0$.
\end{lemma}

\begin{proof}
See Proposition 2.2 in \cite{nttv}.
\end{proof}

\begin{corollary}\label{lem:liouville1}
Let $u$ be a harmonic function in $\R^N$ such that, for some $\alpha\in (0,1)$, there holds
\[
\mathop{\sup_{x,y\in \R^N}}_{x\neq y} \frac{|u(x)-u(y)|}{|x-y|^\alpha}<\infty.
\]
Then u is constant.
\end{corollary}

\begin{lemma}\label{lem:liouville3}
Let $u,v\in H^1_{\rm loc}(\R^N)\cap C(\R^N)$ be nonnegative solutions of the systems
\begin{equation}\label{eq:Liouville_systeminequalities}
\begin{cases}
-\Delta u\leq -\kappa u^p v^{p+1}\\
-\Delta v\leq -\kappa v^pu^{p+1}
\end{cases}
\qquad \text{ in } \R^N,
\end{equation}
with $\kappa>0$ and $p>0$. If
\[
\mathop{\sup_{x,y\in \R^N}}_{x\neq y} \frac{|u(x)-u(y)|}{|x-y|^\alpha}<\infty\quad \text{ and } \quad \mathop{\sup_{x,y\in \R^N}}_{x\neq y} \frac{|v(x)-v(y)|}{|x-y|^\alpha}<\infty,
\]
then either $u\equiv 0$ or $v\equiv 0$.
\end{lemma}

\begin{proof}
For $p\geq 1$, this result is a particular case of Corollary 1.14-(ii) of \cite{SoTe}. Here we present a proof that covers all $p>0$. Initially, we will follow closely the proofs \cite[Lemma 2.5 \& Proposition 2.6]{nttv} and \cite[Section 5]{SoTe}, to which we refer for the complete details. However, at a certain point we will need an extra argument to conclude the case $p<1$.

Let us assume by contradiction that both $u,v\not\equiv 0$. Since $u$ and $v$ are subharmonic, then we have
\begin{equation}\label{eq:bdedawayfrom0}
\frac{1}{r^{n-1}} \int_{\partial B_r}u^2,\ \frac{1}{r^{n-1}}\int_{\partial B_r} v^2 \geq \delta>0 \qquad \text{ for $r$ large}
\end{equation}

\medbreak

\noindent \emph{Step 1.} We define the function
\[
f(r)=\begin{cases}
\frac{2-N}{2}r^2+\frac{N}{2} & \text{ if } r\leq 1\\
\frac{1}{r^{N-2}} & \text{ if } r>0,
\end{cases}
\]
which is $\mathcal{C}^1$ and superharmonic in $\R^N$. For each $r>0$, let $\eta_r$ be the cutoff function such that $0\leq \eta_r\leq 1$, $|\nabla \eta_r|\leq C/r$, $\eta_r=1$ in $B_r$, $\eta_r=0$ in $\R^N\setminus B_{2r}$. By multiplying the first inequality in \eqref{eq:Liouville_systeminequalities} by $\eta^2 f(|x|) u$, and using also the uniform H\"older bounds, we deduce that
\[
\int_{B_r} f(|x|)(|\nabla u|^2+u^{p+1}v^{p+1})\leq C r^{2\alpha}
\]
for large $r>0$ (cf. with \cite[p. 276]{nttv}). Performing an analogue reasoning for the second inequality, we finally conclude that
\[
\int_{B_r} f(|x|)(|\nabla u|^2+u^{p+1}v^{p+1}) \cdot \int_{B_r} f(|x|)(|\nabla v|^2+u^{p+1}v^{p+1})\leq C r^{4\alpha}\ \text{ for large } r>0.
\]

\medbreak

\noindent \emph{Step 2.} Fix $\eps>0$ so that $4\alpha<4-\eps$. We will prove that
\[
J(r):=\frac{1}{r^{4-\eps}} \int_{B_r}f(|x|)(|\nabla u|^2+u^{p+1}v^{p+1}) \cdot \int_{B_r} f(|x|)(|\nabla v|^2+u^{p+1}v^{p+1})
\]
is increasing for $r$ large, which contradicts the conclusion of the previous step.

Using $f(|x|)u$ and $f(|x|)v$ as test functions in \eqref{eq:Liouville_systeminequalities}, we can deduce (compare with \cite[p. 275]{nttv})
\[
\frac{J'(r)}{J(r)}\geq -\frac{4-\eps}{r}+\frac{2\gamma(\Lambda_1(r))}{r}+\frac{2\gamma(\Lambda_2(r))}{r},
\]
where $\gamma(x):=\sqrt{((N-2)/2)^2+x}-(N-2)/2$, and
\[
\Lambda_1(r)=\frac{\int_{\partial B_1} (|\nabla_\theta u_{(r)}|^2+r^2 u_{(r)}^{p+1}v_{(r)}^{p+1})}{\int_{\partial B_1} u_{(r)}^2},\qquad \Lambda_2(r)=\frac{\int_{\partial B_1} (|\nabla_\theta v_{(r)}|^2+r^2 u_{(r)}^{p+1}v_{(r)}^{p+1})}{\int_{\partial B_1} u_{(r)}^2}
\]
for $u_{(r)}(\theta)=u(r\theta)$, $v_{(r)}(\theta)=v(r\theta)$. We recall from \cite[p. 441]{acf} that
\[
\gamma(\lambda_1(A))+\gamma(\lambda_2(B))\geq 2
\]
for every partition of the sphere $S^{N-1}$ in two open sets $A,B$ (here $\lambda_1(E)$ denotes the first Dirichlet eigenvalue on $E\subset S^{N-1}$). We claim that
\begin{equation}\label{eq:contradiction_inequality}
\gamma(\Lambda_1(r))+\gamma(\Lambda_2(r))> \frac{4-\eps}{2},
\end{equation}
which ends this proof. Suppose, in view of a contradiction, that for some $r_n\to \infty$, 
\begin{equation}\label{abs hp 37}
\gamma(\Lambda_1(r_n))+\gamma(\Lambda_2(r_n)) \leq \frac{4-\eps}{2}. 
\end{equation}
Then, in particular, both $\Lambda_1(r_n)$ and $\Lambda_2(r_n)$ are bounded, and
\[
r_n^2 \int_{\partial B_1} u_{(r_n)}^{p+1}v_{(r_n)}^{p+1} \leq C\int_{\partial B_1} u^2_{(r_n)},\ C \int_{\partial B_1} v_{(r_n)}^2.
\]
By multiplying these two inequalities, we deduce that
\[
r_n^2 \int_{\partial B_1} u_{(r_n)}^{p+1}v_{(r_n)}^{p+1} \leq C \| u^2_{(r_n)} \|_{L^2(\partial B_1)}\| v^2_{r_n}\|_{L^2(\partial B_1)}\leq C' \| u^2_{(r_n)} \|^{p+1}_{L^2(\partial B_1)}\| v^2_{r_n}\|^{p+1}_{L^2(\partial B_1)},
\]
where the last inequality comes from \eqref{eq:bdedawayfrom0}. As a consequence, recalling also \eqref{abs hp 37}, the normalised functions 
\[
\tilde u_{n}= \frac{u_{(r_n)}}{\|u_{(r_n)}\|_{L^2(\partial B_1)}},\qquad \tilde v_{n}= \frac{v_{(r_n)}}{\|v_{(r_n)}\|_{L^2(\partial B_1)}}
\]
are uniformly bounded in $H^1(\partial B_1)$, and
\[
r_n^2\int_{\partial B_1} \tilde u_{n}^{p+1} \tilde v_{n}^{p+1} \leq C.
\]
Thus, up to a subsequence, $\tilde u_n\rightharpoonup \tilde u$, $\tilde v_n\rightharpoonup \tilde v$ weakly in $H^1(\partial B_1)$, with $\tilde u\cdot \tilde v\equiv 0$. This, in turn, gives:
\[
2>\frac{4-\eps}{2}\liminf_n \gamma(\Lambda_1(r_n))+\gamma(\Lambda_2(r_n))\geq \gamma(\lambda_1(\{\tilde u>0\}))+\gamma(\lambda_1(\{\tilde v>0\}))\geq 2,
\]
a contradiction.
\end{proof}


\end{document}